\documentclass[11pt, leqno]{amsart}
\usepackage{amsfonts,amssymb, amscd}
\usepackage{amsmath}
\usepackage{lmodern}
\usepackage{mathrsfs} 
\usepackage{amsthm}
\usepackage{qsymbols}
\usepackage{mathtools}
\mathtoolsset{showonlyrefs}
\usepackage{dsfont}
\usepackage{graphicx}
\usepackage{latexsym}
\usepackage{chngcntr}
\usepackage[noadjust]{cite}
\usepackage{paralist}
\usepackage[parfill]{parskip}
\usepackage{bm}
\usepackage{tikz-cd}
\usepackage{csquotes}
\usetikzlibrary{decorations.pathreplacing}

\usepackage{enumitem}
\newtheorem{theorem}{Theorem}[section]
\newtheorem{lemma}[theorem]{Lemma}
\newtheorem{proposition}[theorem]{Proposition}
\newtheorem{corollary}[theorem]{Corollary}
\newtheorem{assumption}[theorem]{Assumption}
\theoremstyle{definition}
\newtheorem{definition}[theorem]{Definition}
\newtheorem{remark}[theorem]{Remark}
\newtheorem{example}[theorem]{Example}
\usepackage[plainpages=false,pdfpagelabels,
citecolor=red]{hyperref}
\usepackage{xcolor}
\hypersetup{
 colorlinks,
 linkcolor={cyan!90!black},
 citecolor={magenta},
 urlcolor={green!40!black}
} 
\newcommand{\R}{\mathbb{R}}
\newcommand{\IC}{\mathbb{C}}
\newcommand{\IN}{\mathbb{N}}
\newcommand{\IZ}{\mathbb{Z}}

\newcommand{\IL}{\mathbb{L}}

\newcommand{\IX}{\mathbb{X}}
\newcommand{\IW}{\mathbb{W}}
\newcommand{\Res}{\mathcal{R}}
\newcommand{\Ext}{\mathcal{E}}
\newcommand{\Pro}{\mathcal{P}}
\newcommand{\cF}{\mathcal{F}}

\newcommand{\cH}{\mathcal{H}}
\newcommand{\cD}{\mathcal{D}}
\newcommand{\cI}{\mathcal{I}}
\let\SS\S	
\renewcommand{\S}{\mathcal{S}}
\newcommand{\AS}{\mathcal{AS}}
\renewcommand{\L}{\mathrm{L}}
\newcommand{\C}{\mathrm{C}}
\newcommand{\B}{\mathrm{B}}
\renewcommand{\H}{\mathrm{H}}
\newcommand{\W}{\mathrm{W}}
\newcommand{\F}{\mathrm{F}}
\newcommand{\X}{\mathrm{X}}
\newcommand{\e}{\mathrm{e}}

\renewcommand{\d}{\mathrm{d}}
\renewcommand{\r}{\mathrm{r}}

\newcommand{\eps}{\varepsilon}

\newcommand{\bracket}{\langle \cdot\,,\cdot\rangle}

\newcommand{\cl}[1]{\overline{#1}}

\DeclareMathOperator{\bd}{\partial \!}
\DeclareMathOperator{\supp}{supp}

\DeclareMathOperator{\dist}{d}
\DeclareMathOperator{\diam}{diam}

\DeclareMathOperator{\Id}{1}
\DeclareMathOperator{\Rg}{\mathsf{R}}
\DeclareMathOperator{\Ke}{\mathsf{N}}
\DeclareMathOperator{\dom}{\mathsf{D}}

\def\XXint#1#2#3{{\setbox0=\hbox{$#1{#2#3}{%
\int}$ }
\vcenter{\hbox{$#2#3$ }}\kern-.6\wd0}}
\title[Interpolation for Sobolev functions with partially vanishing trace] {Interpolation theory for Sobolev functions with partially vanishing trace on irregular open sets}
\author{Sebastian Bechtel}
\address{Fachbereich Mathematik, Technische Universit\"at Darmstadt, Schlossgartenstr. 7, 64289 Darmstadt, Germany}
\email{bechtel@mathematik.tu-darmstadt.de}
\author{Moritz Egert}
\address{Laboratoire de Math\'{e}matiques d'Orsay, Univ.\ Paris-Sud, CNRS, Universit\'{e} Paris-Saclay, 91405 Orsay, France}
\email{moritz.egert@math.u-psud.fr}
\date{\today}
\dedicatory{}
\thanks{}
\begin{document}
\begin{abstract}
A full interpolation theory for Sobolev functions with smoothness between $0$ and $1$ and vanishing trace on a part of the boundary of an open set is established. Geometric assumptions are mostly of measure theoretic nature and reach beyond Lipschitz regular domains. Previous results were limited to regular geometric configurations or Hilbertian Sobolev spaces. Sets with porous boundary and their characteristic multipliers on smoothness spaces play a major role in the arguments.
\end{abstract}
\maketitle
\section{Introduction and main results}
\label{Sec: Introduction}

Recent years have witnessed an ever-growing interest in the treatment of quasilinear equations of parabolic type through maximal regularity techniques~\cite{DHP}. To a large extend this stems from the flexibility of the approach and its applicability to rough geometric configurations that arise in applications, for example reaction-diffusion models related to differential operators or systems in divergence form with mixed boundary conditions~\cite{Bonifacius-Neitzel, Disser, Disser-Meyries-Rehberg, Karo-Joachim, HJKR, RobertJDE, Termistor1, Termistor2, Rehberg-terElstADE}. There, the linear second order operator usually admits a proper theory of weak solutions on Sobolev spaces of type $\W^{1,p}$ carrying the boundary conditions, but in controlling non-linear terms of high order and reaction processes on lower-dimensional substructures of the boundary simultaneously, interpolation spaces of order $s \in (0,1)$ are most appropriate. We refer to \cite[Sec.~4.1]{Karo-Joachim} for a detailed account on this paradigm in the context of dynamics in a semiconductor device, see also \cite[Sec.~6]{RobertJDE}. This leads to the problem of identifying such interpolation spaces in the presence of rough boundaries to the fractional Sobolev- and Bessel potential spaces, in analogy with what is long known for function spaces on $\R^d$ or on smooth domains with pure Dirichet boundary condition~\cite{Triebel, BL, Seeley}. 

At the heart of the matter lies the following question. Given an open set $O \subseteq \R^d$ and a piece $D \subseteq \bd O$ of its boundary, define the Sobolev space $\W^{1,p}_D(O)$ as the $\W^{1,p}(O)$-closure of smooth functions whose support stays away from $D$. Under which geometric assumptions can one determine explicitly the interpolation spaces
\begin{align*}
 [\L^p(O), \W^{1,p}_D(O)]_s \quad \text{and} \quad (\L^p(O), \W^{1,p}_D(O))_{s,p}
\end{align*}
defined through Calder\'on--Lions' complex and Peetre's real method? The space $\W^{1,p}_D(O)$ should be thought of the collection of $\W^{1,p}(O)$-functions with homogeneous Dirichlet boundary condition on $D$.

Interpolation theory related to the spaces $\W_D^{1,p}(O)$ has recently been studied in \cite{AKM, TripleMitrea-Brewster, ABHR, HJKR, Griepentrog-InterpolationOnGroger, Darmstadt-KatoMixedBoundary}, but mostly with a focus on interpolating with respect to integrability. Interpolation in differentiability appears only in \cite{Darmstadt-KatoMixedBoundary, AKM} for $p=2$ and in \cite{Griepentrog-InterpolationOnGroger} for general $p$ on certain model sets. The main difficulty lies in that taking the boundary trace on $D$ can, if at all, be defined in a meaningful way only on the Sobolev space~\cite[Sec.~6.6]{BL}. This forbids to treat the question via purely functorial techniques.

We close this gap by establishing a full interpolation theory under geometric assumptions in the spirit of what has become standard for treating mixed boundary value problems~\cite{ABHR, TripleMitrea-Brewster, HJKR, Rehberg-terElstADE}. In particular, we confirm the formula for the complex interpolation spaces that was conjectured in connection with fractional powers of divergence form operators in \cite[Rem.~10.5]{ABHR} and listed as an open problem in \cite[Sec.~5.3]{Karo-Joachim}. We also treat interpolation simultaneously in differentiability and integrability. Some of our results appear to be new even on much more regular domains since we do not require that the interface of $D$ with the complementary boundary part $\bd O \setminus D$ can be parametrized by coordinate charts in any sense.

\subsection{Geometric setting}
\label{Subsec: Intro geometry}

We shall work on open sets $O \subseteq \R^d$, not necessarily connected or bounded, satisfying the thickness condition
\begin{align}
\label{eq: thickness condition}
 c \leq \frac{|B \cap O|}{|B|} \leq C
\end{align}
for some constants $0<c\leq C < 1$ and all balls $B$ of radius $\r(B) \leq 1$ centered at the boundary $\bd O$. This excludes that $O$ have interior or exterior cusps. We assume that the Dirichlet part $D \subseteq \bd O$ is a $(d-1)$-regular, not necessarily closed set in the sense of Jonsson--Wallin~\cite{JW}. Only around the complementary boundary part $\cl{\bd O \setminus D}$ we demand Lipschitz coordinate charts with uniformly controlled bi-Lipschitz constants, which on domains with compact boundary reduces to the usual weak Lipschitz condition. Finally, the interface $\bd D$ between the two boundary parts should be a \emph{porous} subset of the full boundary. This means that there should exist some $\kappa \in (0,1)$ with the property that every ball $B$ of radius $\r(B) \leq 1$ centered in $\bd D$ contains a ball of radius $\kappa r$ centered in $\bd O$ that avoids $\bd D$. A related condition appeared in~\cite{Rehberg-terElstADE}.

Porosity plays a fundamental role in our considerations and for the reader's convenience we have collect some folklore on this concept in Appendix~\ref{Sec: Porous sets}. We often take advantage of it in form of equivalent but less transparent conditions related to Aikawa- and Assouad dimension. In particular, all our results hold if $\bd D$ is $(d-2)$-regular as in Figure~\ref{fig: example}.
\begin{figure}[ht]
\label{Fig: Cone}
\centering
    \includegraphics[scale=1.0]{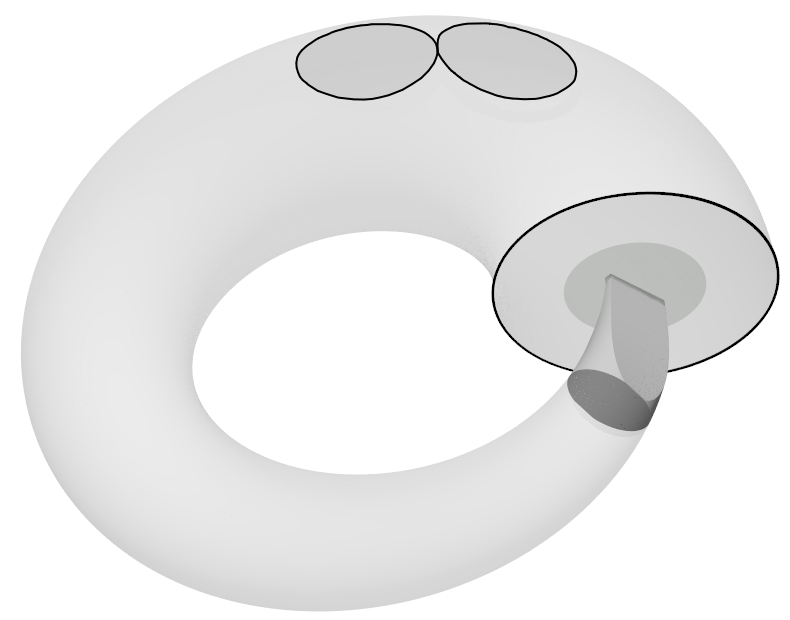}
    \caption{The domain $O \subseteq \R^3$ is obtained by transforming a cylinder such that one lateral boundary part degenerates to a line segment touching the opposed side from outside. The dark-shaded boundary parts carry the Dirichlet condition.}
\label{fig: example}
\end{figure}
We believe that this setting is rather common in applications. It includes the Gr\"oger regular sets~\cite{Groeger} used for instance in \cite{RobertJDE, Termistor1, Termistor2, Bonifacius-Neitzel}. Our interpolation results are new even in this context since compared to earlier work~\cite{Griepentrog-InterpolationOnGroger} we remove the requirement that the Lipschitz coordinate charts should be measure preserving.

\subsection{Main results}
\label{Subsec: Intro results}

Fractional Sobolev spaces $\W^{s,p}$ and Bessel potential spaces $\H^{s,p}$ on $O$, with and without Dirichlet boundary conditions on $D$ indicated by a subscript, are properly defined in Sections~\ref{Subsec: Not Dirichlet} - \ref{Subsec: Not spaces on O} through restriction from $\R^d$. The Dirichlet condition is understood for $s>1/p$ in virtue of Jonsson--Wallin's trace theory~\cite{JW}. For integer $s$ we have $\W^{s,p} = \H^{s,p}$ up to equivalent norms. In particular, we retrieve from $\L^p = \H^{0,p}$ and $\W^{1,p} = \H^{1,p}$ answers to the interpolation questions raised above.

In most interpolation results we shall have possibly different Lebesgue exponents $p_0,p_1 \in (1,\infty)$ and smoothness parameters $s_0, s_1 \in \R$ and we interpolate in both scales simultaneously. In order to straighten the presentation, we introduce here, given $\theta \in (0,1)$, the interpolating parameters $p \in (1,\infty)$ and $s \in \R$ through
\begin{align}
\label{eq: interpolating paramaters}
 \frac{1}{p} \coloneqq \frac{1-\theta}{p_0} + \frac{\theta}{p_1}, \qquad s \coloneqq (1-\theta)s_0 + \theta s_1.
\end{align}
In the presence of $p_i$, $s_i$ and $\theta$ as above we shall \emph{exclusively} use the symbols $p$ and $s$ in that very sense, sometimes without further mentioning.

Let us state the central result of our paper. The proof is given in Section~\ref{Sec: Complex}, including an informal outline in Section~\ref{Subsec: Roadmap complex}.  

\begin{theorem}
\label{thm: main result}
Assume the geometric setting of Section~\ref{Subsec: Intro geometry}. Let $p_0,p_1 \in (1,\infty)$, $s_0 \in [0,1/p_0)$, $s_1 \in (1/p_1,1]$, and for $\theta \in (0,1)$ define $p$ and $s$ as in \eqref{eq: interpolating paramaters}. If $\X$ denotes either $\H$ or $\W$, then the complex interpolation identity
\begin{align}
 \tag{a}\label{C3} [\X^{s_0,p_0}(O), \X_D^{s_1,p_1}(O)]_{\theta} &= \begin{cases} 
                                                                 \X_D^{s,p}(O) &(\text{if $s>1/p$})\\ \X^{s,p}(O)  &(\text{if $s<1/p$})
                                                                \end{cases}
\intertext{holds up to equivalent norms as well as the real interpolation identity}
\tag{b}\label{C4} \noeqref{C4} (\X^{s_0,p_0}(O), \X_D^{s_1,p_1}(O))_{\theta,p} &= \begin{cases} 
                                                                 \W_D^{s,p}(O) &(\text{if $s>1/p$})\\ \W^{s,p}(O)  &(\text{if $s<1/p$})                                                       
                                                                 \end{cases}
\end{align}
with the exception that $s_0 \neq 0$ and $s_1 \neq 1$ are required in \eqref{C3} for $\X = \W$.
\end{theorem}

Interpolation theory for the spaces $\X^{s,p}(O)$ without boundary conditions becomes apparent from an extension result of Rychkov~\cite{Rychkov} that we shall review in Section~\ref{Subsec: Not spaces on O} and make extensive use of. We establish further structural properties of our function spaces in Section~\ref{Sec: First properties}. Abstract techniques then lead us in Section~\ref{Subsec: Symmetric interpolation} to the following interpolation results for two function spaces with Dirichlet condition. This only requires $O$ and $D$ to be regular in the sense of Jonsson--Wallin~\cite{JW}.

\begin{theorem}
\label{thm: main result large s}
Let $O \subseteq \R^d$ be an open, $d$-regular set, and let $D \subseteq \cl{O}$ be $(d-1)$-regular. Let $p_0,p_1 \in (1,\infty)$, $s_0 \in (1/p_0,1+1/p_0)$, $s_1 \in (1/p_1,1+1/p_1)$, and for $\theta \in (0,1)$ define $p$ and $s$ as in \eqref{eq: interpolating paramaters}. Let $\X$ denote either $\H$ or $\W$. Up to equivalent norms it follows that
\begin{align*}
 \tag{c}\label{C1} [\X_D^{s_0,p_0}(O), \X_D^{s_1,p_1}(O)]_{\theta} &= \X_D^{s,p}(O), \\[8pt]
 \tag{d}\label{C2} (\X_D^{s_0,p_0}(O), \X_D^{s_1,p_1}(O))_{\theta,p} &= \W_D^{s,p}(O),
\end{align*}
with the two exceptions that in \eqref{C1} for $\X = \W$ either all or none of $s_0,s_1,s$ have to be $1$ and that in \eqref{C2} the value $s = 1$ is only permitted when $s_0 = s_1 = 1$.
\end{theorem}

As a cautionary tale, let us remark that \emph{a priori} all function spaces are defined by restrictions. In particular, $\W^{1,p}(O) = \H^{1,p}(O)$ might be smaller than the collection of $\L^p(O)$-functions whose first-order distributional derivatives are in $\L^p(O)$ under the assumptions of Theorem~\ref{thm: main result large s}. Under the full set of geometric assumptions in Theorem~\ref{thm: main result}, however, there is no such ambiguity. We have collected further background on this issue in Appendix~\ref{Sec: Intrinsic characterizations}.

\subsection{Extensions and generalizations}
\label{Subsec: Intro applications}

Abstract reiteration and duality theorems~\cite{Triebel, BL, Wolff-refined} imply numerous further interpolation results that invoke our Theorems~\ref{thm: main result} and~\ref{thm: main result large s} ``off-the-shelf''. We leave the care of writing them down to the interested readers. Here, we only present one such result that turned out useful in the $\W^{-1,p}$-theory of divergence form operators~\cite{Disser, Termistor1, Termistor2} and previously was available only in the restrictive setup of \cite[Lemma~3.4]{Griepentrog-InterpolationOnGroger}. The proof of this result will be given in Section~\ref{Sec: W-1p}. We write $\W^{-1,p}_D(O)$ for the space of conjugate linear functionals on $\W^{1,p'}_D(O)$, where $1/p + 1/p' = 1$. 

\begin{theorem}
\label{thm: W1p-W-1p}
Assume the geometric setting of Section~\ref{Subsec: Intro geometry} and let $p \in (1,\infty)$. Up to equivalent norms it follows that
\begin{align}
    \tag{e}\label{C5} \noeqref{C5}
[\W^{-1,p}_D(O), \W_D^{1,p}(O)]_{1/2} = \L^p(O).
\end{align}
\end{theorem}

In Section~\ref{Sec: Real} we present a method tailored for real interpolation of fractional Sobolev spaces with the same integrability that is in some sense dual to the proof of Theorem~\ref{thm: main result}. It bears the advantage of a more general geometric setting. 

\begin{theorem}
\label{thm: real interpolation via trace}
Let $O \subseteq \R^d$ be an open, $d$-regular set with $(d-1)$-regular boundary, and let $D \subseteq \cl{O}$ be uniformly $(d-1)$-regular. Let $p \in (1,\infty)$, $s_0 \in [0, 1/p)$, $s_1 \in (1/p,1]$, and $\theta \in (0,1)$. Up to equivalent norms it follows that
\begin{align}
\tag{f}\label{C6}
(\W^{s_0,p}(O), \W_D^{s_1,p}(O))_{\theta, p} &= \begin{cases} 
                                               \W_D^{s,p}(O) &(\text{if $s>1/p$})\\ \W^{s,p}(O)  &(\text{if $s<1/p$})
					      \end{cases},
\end{align}
where $s\coloneqq (1-\theta)s_0 + \theta s_1$.
\end{theorem}
Our proof simplifies \cite[Sec.~7]{Darmstadt-KatoMixedBoundary}, where the case $p=2$ was treated on bounded domains with a Lipschitz assumption around $\cl{\bd O \setminus D}$. Uniform $(d-1)$-regularity is defined in the next section. For bounded sets there is no difference with $(d-1)$-regularity. 

\subsection*{Acknowledgment}
Both authors are grateful to Joachim Rehberg for many fruitful discussions on and around the topic. The first named author thanks his Ph.D.\ advisor Robert Haller-Dintelmann for his support and the Laboratoire de Math\'{e}matiques d'Orsay for hospitality during a stay in March 2018 where this project got started.
\section{Notation and background}
\label{Sec: Notation and background}

\subsection{Geometry}
\label{Subsec: Not Geometry}

We work in Euclidean space $\R^d$ of dimension $d \geq 2$. We write $B = \B(x,r)$ for the open ball of radius $\r(B) = r$ centered at $x \in \R^d$ and $cB$ for the concentric ball of radius $c \r(B)$. By $\cH^\ell$ we denote the $\ell$-dimensional \emph{Hausdorff measure} in $\R^d$ \cite[Ch.~7]{Yeh} and by $|\cdot|$ the Lebesgue measure in $\R^d$. We write $\diam(\,\cdot\,)$ and $\dist(\cdot \,,\cdot)$ for the \emph{diameter} and the \emph{(semi) distance} of sets induced by the Euclidean distance on $\R^d$.

\begin{definition}
\label{def: Lipschitz property}
An open set $O \subseteq \R^d$ satisfies a \emph{uniform Lipschitz condition} around a closed boundary part $F \subseteq \bd O$ if the following holds. For every $x \in F$ there is an open neighborhood $U_x \ni x$ and a bi-Lipschitz transformation $\Phi_x: U_x \to (-1,1)^d$ such that $\Phi_x(x) = 0$ and
\begin{align}
\label{eq: Lipschitz property}
 \Phi_x(U_x \cap O) = (0,1) \times (-1,1)^{d-1}, \qquad \Phi_x(U_x \cap \partial O) = \{0\} \times (-1,1)^{d-1},
\end{align}
and there exists a number $L$ that bounds the bi-Lipschitz constants of all $\Phi_x$. \emph{Bi-Lipschitz constant} refers to the maximum of the Lipschitz constants of $\Phi_x$ and $\Phi_x^{-1}$.
\end{definition}

\begin{remark}
\label{rem: Lipschitz property for compact boundary}
If $F$ is compact, then the uniform Lipschitz condition is equivalent to requiring existence of bi-Lipschitz maps $\Phi_x$ with \eqref{eq: Lipschitz property} for every $x \in F$. This so-called \emph{weak Lipschitz condition} is weaker than requiring that $O$ has a Lipschitz boundary near $x$, see \cite[Sec.~7.3]{RobertJDE} for a relevant example. 

To see the equivalence, we cover $F$ by finitely many corresponding neighborhoods. Let $L_0$ be the maximal bi-Lipschitz constant and fix a Lebesgue number $\delta>0$, that is, for $x \in F$ we have $B(x,\delta) \subseteq U$ for one pre-selected neighborhood $U$ with bi-Lipschitz transformation $\Phi$. Suitable dilations yield a bi-Lipschitz map $T_x$ of the unit cube onto itself that preserves the upper and lower half and maps $\Phi(x)$ to $0$. Due to $B(\Phi(x),\delta/L_0) \subseteq  (-1,1)^{d}$ its bi-Lipschitz constant is controlled by $\delta$ and $L$. So, we can take $U_x \coloneqq U$ and $\Phi_x \coloneqq T_x \circ \Phi$.
\end{remark}

\begin{definition}
\label{def: l-set}
A set $E \subseteq \R^d$ is called \emph{$\ell$-Ahlfors regular} or simply \emph{$\ell$-regular}, if there is comparability
\begin{align}
\label{eq: l-set}
 \cH^\ell(B \cap E) \approx \r(B)^\ell
\end{align}
uniformly for all open balls $B$ of radius $\r(B) \leq 1$ centered in $E$. If comparability holds for $\r(B) \leq \diam(E)$, then $E$ is called \emph{uniformly $\ell$-regular}.
\end{definition}

We remark that $\ell$-regular sets are $\ell$-sets in the sense of Jonsson--Wallin~\cite[Thm.~II.1.1]{JW}. Uniformly $\ell$-regular sets are $\ell$-regular and the converse holds for bounded sets, see Lemma~\ref{lem: Ahlfors radius bound}. Many authors consider only \emph{closed} regular sets, but most considerations adapt \emph{verbatim} since in the situation above the closure $\cl{E}$ is still $\ell$-regular and $\cl{E} \setminus E$ is an $\cH^\ell$ null set \cite[Prop.~VIII.1]{JW}. We shall frequently  use this result without further reference.

\begin{example}
\label{ex: d-set}
If $O \subseteq \R^d$ satisfies the thickness condition \eqref{eq: thickness condition}, then $O$ and ${}^cO$ are both $d$-regular. Indeed, by symmetry of \eqref{eq: thickness condition} it suffices to check $d$-regularity of $O$. For $B$ a ball centered in $O$ we distinguish whether or not the concentric ball of half the radius intersects $\bd O$. If it does not, $\cH^d(B \cap O) \approx \r(B)^d$ is obvious and if it does, the same is guaranteed by \eqref{eq: thickness condition} since $B$ contains a ball of radius $\r(B)/2$ centered in $\bd O$.
\end{example}

\begin{example}
\label{ex: l-set}
If, as in Section~\ref{Subsec: Intro geometry}, $D$ is a $(d-1)$-regular part of the boundary of an open set $O \subseteq \R^d$ that satisfies the uniform Lipschitz condition around $\cl{\bd O \setminus D}$, then the full boundary $\bd O$ is also $(d-1)$-regular. Indeed, we can use that bi-Lipschitz images have comparable $\cH^{d-1}$-measure~\cite[Thm.~28.10 a)]{Yeh} and that the bi-Lipschitz constants are uniformly bounded to show that $\cl{\bd O \setminus D}$ is $(d-1)$-regular. We conclude by the observation that the class of $(d-1)$-regular sets is closed under finite unions.
\end{example}

We recall with slight modification the notion of porous sets introduced by Vais\"al\"a~\cite{Vaisala-porous}.

\begin{definition}
\label{def: porosity}
Let $E \subseteq F \subseteq \R^d$. Then $E$ is \emph{porous in $F$} if there exists a constant $\kappa \in (0,1]$ with the following property:
\begin{align}
\label{eq: Def porosity}
 \forall x \in E, r \leq 1 \quad \exists y \in \B(x,r) \cap F \; : \; \B(y, \kappa r) \cap E = \emptyset.
\end{align}
If this holds for all $r \leq \diam(E)$, then $E$ is called \emph{uniformly porous in $F$}. If $F = \R^d$, then $E$ is simply called \emph{(uniformly) porous}. 
\end{definition}

\begin{remark}
\label{rem: porosity}
If $E$ is uniformly porous with constant $\kappa$, then it is porous with constant $\min\{\kappa, \kappa \diam(E)\}$. Condition \eqref{eq: Def porosity} implies the seemingly stronger statement 
\begin{align*}
 \forall x \in F, r \leq 1 \quad \exists y \in \B(x,r) \cap F \; : \; \B(y, \kappa r/4) \subseteq \B(x,r) \setminus E.
\end{align*}
This is seen by distinguishing whether or not $B(x,r/2)$ intersects $E$. An analogous remark applies to uniformly porous sets.
\end{remark}

Whenever necessary, the reader can refer to Appendix~\ref{Sec: Porous sets} for further background on porous sets. It is instructive to think of them as lower dimensional than the ambient space. This is made precise in Proposition~\ref{prop: porosity through assouad} via the following notions of Assouad dimension.

\begin{definition}
\label{def: Assouad dimension}
Let $E \subseteq \R^d$. Let $\overline{\AS}(E)$ denote the set of $\lambda > 0$ for which there exists $C \geq 0$ such that, if $0<r<R< 2\diam(E)$ and $x \in E$, then at most $C(R/r)^\lambda$ balls of radius $r$ centered in $E$ are needed to cover $E \cap\B(x,R)$. The number $\overline{\dim}_{\AS}(E) \coloneqq \inf \overline{\AS}(E)$ is called \emph{upper Assouad dimension} of $E$. The corresponding \emph{lower Assouad dimension} is defined as $\underline{\dim}_{\AS}(E) \coloneqq \sup \underline{AS}(E)$ with $\underline{AS}(E)$ the set of $\lambda>0$ for which there exists $C \geq 0$ such that in the former situation at least $C(R/r)^\lambda$ balls are needed.
\end{definition}

There is no ambiguity with uniformly $\ell$-regular sets since their dimension is $\ell$ for any of these concepts, see Proposition~\ref{prop: Assouad for Ahlfors}. The reader can readily check that Definitions~\ref{def: l-set}, \ref{def: porosity}, and \ref{def: Assouad dimension} are purely topological in that they do not change when replacing open by closed balls and/or balls by axis-aligned cubes.

\subsection{Banach spaces and interpolation}
\label{Subsec: Not Interpolation}

All Banach spaces are over the complex numbers. We assume some familiarity with real and complex interpolation of Banach spaces and refer to the textbooks \cite{BL, Triebel} for background. However, understanding this paper does not require the precise construction of interpolation spaces. We shall only need the general methodology, their fundamental properties, and standard results on interpolation of function spaces on $\R^d$ measuring smoothness to be recalled further below.

Let $(X_0,X_1)$ be an \emph{interpolation couple}, that is, a pair of Banach spaces that are included in a common linear Hausdorff space. Then the following Banach spaces can be defined between $X_0 \cap X_1$ and $X_0 + X_1$ with respect to continuous inclusion: For $\theta \in [0,1]$ the \emph{complex interpolation spaces} $[X_0, X_1]_\theta$ of Calder\'on--Lions~\cite[Sec.~4.1]{BL} and for $\theta \in (0,1)$ and $p \in [1,\infty]$ the \emph{real interpolation spaces} $(X_0,X_1)_{\theta,p}$ obtained from Peetre's $K$-method~\cite[Sec.~3.1]{BL}. In any of these spaces $X_0 \cap X_1$ is dense~\cite[Thm.~3.4.2 \& 4.4.2]{BL}. In particular, the endpoints $[X_0,X_1]_{j}$, $j \in \{0,1\}$, coincide with $X_j$ only if $X_0 \cap X_1$ is dense in $X_j$.

\subsection{Function spaces measuring smoothness}
\label{Subsec: Not Spaces}

We define the relevant function spaces of smoothness $s \in \R$ and integrability $p\in (1,\infty)$. The \emph{Bessel potential space} $\H^{s,p}(\R^d)$ consists of those tempered distributions $f \in \S'(\R^d)$ for which
\begin{align*}
 \|f\|_{\H^{s,p}} \coloneqq \|\cF^{-1}(1+|\cdot|^2)^{s/2} \cF f\|_{\L^p} < \infty.
\end{align*}
Here, $\cF$ denotes the Fourier transform. With $1/p' \coloneqq 1 - 1/p$ the spaces $\H^{-s,p}(\R^d)$ and $\H^{s,p'}(\R^d)$ are in a sesquilinear duality extending the $\L^2$ inner product~\cite[Sec.~2.4.2]{Triebel}.

If $k \geq 0$ is an integer, then $\H^{k,p}(\R^d)$ coincides up to equivalent norms with the \emph{Sobolev space} $\W^{k,p}(\R^d)$ of tempered distributions such that
\begin{align*}
 \|f\|_{\W^{k,p}} \coloneqq \bigg(\|f\|_{\L^p}^p + \sum_{j=1}^d \|\partial_j^k f\|_{\L^p}^p \bigg)^{1/p} < \infty,
\end{align*}
where we could have equivalently taken all derivatives up to order $k$ into account, see \cite[Sec.~2.3.3]{Triebel}. Note that we have $\H^{0,p}(\R^d) = \W^{0,p}(\R^d) = \L^p(\R^d)$. 

For $s = k + \sigma$ with $k \geq 0$ an integer and $\sigma \in (0,1)$ there is a \emph{fractional Sobolev space} $\W^{s,p}(\R^d)$ defined through requiring 
\begin{align*}
 \|f\|_{\W^{s,p}(\R^d)} \coloneqq \|f\|_{\W^{k,p}(\R^d)} + \bigg(\sum_{j=1}^d \iint_{\R^d \times \R^d} \frac{|\partial_j^k f(x) - \partial_j^k f(y)|^p}{|x-y|^{d+\sigma p}} \; \d x \; \d y \bigg)^{1/p} < \infty.
\end{align*}
We could also have restricted integration to $|x-y|<1$. Further common equivalent norms on these spaces exist~\cite[Sec.~2.5.1]{Triebel}. We define $\W^{-s,p}(\R^d)$ as the space of conjugate-linear functionals on $\W^{s,p'}(\R^d)$ in accordance with what we have seen for Bessel potential spaces. We could have also given an equivalent intrinsic definition using the scale of Besov spaces~\cite[Sec.~2.3.2/2.6.1]{Triebel} but the view point of dual spaces is better suited to our circumstances.

We collect all interpolation properties proved in \cite[Sec.~2.4.2]{Triebel} that shall be used ``off-the-shelf'' in the further course. In \cite{Triebel} the nomenclature is $\H^{s,p} = \F^{s}_{p,2}$ and $\W^{s,p} = \F^{s}_{p,p}$ for non-integer $s$.

\begin{proposition}
\label{prop: interpolation whole space}
Let $p_0, p_1 \in (1,\infty)$, $s_0,s_1 \in \R$, and $\theta \in (0,1)$. Let $\X$ denote either $\H$ or $\W$. Up to equivalent norms it follows that
\begin{align}
 \tag{i}\label{Ii} 	[\X^{s_0, p_0}(\R^d), \X^{s_1, p_1}(\R^d)]_{\theta} &= \X^{s,p}(\R^d), \\[8pt]
 \tag{ii}\label{Iii} \noeqref{Iii} 	(\X^{s_0, p_0}(\R^d), \X^{s_1, p_1}(\R^d))_{\theta,p} &= \W^{s,p}(\R^d),
\end{align}
with the two exceptions that in \eqref{Ii} for $\X = \W$ either all or none of $s_0,s_1,s$ have to be integers and that in \eqref{Iii} integer $s$ is only permitted when $s_0 = s_1 (=s)$.
\end{proposition}

\subsection{Function spaces on \texorpdfstring{$\R^d$}{Rd} incorporating a Dirichlet condition}
\label{Subsec: Not Dirichlet}

We define analogous spaces of functions with positive smoothness on $\R^d$ that vanish on some $(d-1)$-regular set $E \subseteq \R^d$. All this is based on celebrated results of Jonsson--Wallin~\cite{JW}. 

We need the notion of fractional Sobolev spaces on $E$. They are denoted $\B_s^{p,p}(E)$ in \cite{JW} but to keep the analogy with the previous section we shall write $\W^{s,p}(E)$ instead. Having equipped $E$ with $(d-1)$-dimensional Hausdorff measure, we define for $s\in (0,1)$ and $p \in (1,\infty)$ this space as the Banach space of those $f \in \L^p(E)$ for which
\begin{align*}
 \|f\|_{\W^{s,p}(E)} \coloneqq &\bigg(\int_E |f(x)|^p \; \cH^{d-1}(\d x) \bigg)^{1/p} \\
 &+ \bigg(\iint_{\substack{x,y \in E \\ |x-y| < 1}} \frac{|f(x) - f(y)|^p}{|x-y|^{d-1 + sp}} \; \cH^{d-1}(\d x) \; \cH^{d-1}(\d y) \bigg)^{1/p} < \infty.
\end{align*}
If $E$ is closed, the following is proved in \cite[Thm.~VI.1 \& VII.1]{JW}. The general case follows from the discussion in Section~\ref{Subsec: Not Geometry}.

\begin{proposition}
\label{prop: JW}
Suppose $E \subseteq \R^d$ is $(d-1)$-regular. Let $p\in (1,\infty)$, $s \in (1/p, 1+1/p)$, and let $\X$ denote either $\H$ or $\W$.
\begin{enumerate}
 \item If $f \in \X^{s,p}(\R^d)$, then for $\cH^{d-1}$-almost every $x \in E$ the limit
 \begin{align*}
  (\Res_E f)(x) \coloneqq \lim_{r\to 0} \frac{1}{|\B(x,r)|} \int_{\B(x,r)} f(y) \; \d y
 \end{align*}
 exists. The restriction operator $\Res_E$ maps $\X^{s,p}(\R^d)$ boundedly into $\W^{s-1/p,p}(E)$. 
 
 \item Conversely, there exists a bounded extension operator $\Ext_E: \W^{s-1/p,p}(E) \to \X^{s,p}(\R^d)$ that serves as a right inverse for $\Res_E$. It does not depend on $p$, $s$.
\end{enumerate}
\end{proposition}

We often refer to $\Res_E$ and $\Ext_E$ as the \emph{Jonsson--Wallin operators for $E$}.

\begin{definition}
\label{def: XspE}
Let $E \subseteq \R^d$ be $(d-1)$-regular. Given $p\in (1,\infty)$ and $s \in (1/p, 1+1/p)$, define
\begin{align*}
 \X^{s,p}_E(\R^d) \coloneqq \{ f \in \X^{s,p}(\R^d) : \Res_E f = 0 \},
\end{align*}
where $\X$ denotes either $\H$ or $\W$, and equip it with the norm inherited from $\X^{s,p}(\R^d)$.
\end{definition}

\subsection{Function spaces on open sets with and without partially vanishing trace}
\label{Subsec: Not spaces on O}

As usual, let $\X$ denote either $\H$ or $\W$. Since for $s \geq 0$ we have $\X^{s,p}(\R^d) \subseteq \L^p(\R^d)$, the \emph{pointwise restriction} $|_O$ of functions to $O$ is defined on $\X^{s,p}(\R^d)$.

\begin{definition}
\label{def: spaces on open sets}
Let $O \subseteq \R^d$ be an open set and let $s \geq 0$, $p \in (1,\infty)$. Define $\X^{s,p}(O) \coloneqq\{f|_O : f \in \X^{s,p}(\R^d)\}$ with quotient norm
\begin{align*}
 \|f\|_{\X^{s,p}(O)} \coloneqq \inf \big\{\|F\|_{\X^{s,p}} : F \in \X^{s,p}(\R^d) \text{ and } F|_O = f \big\}.
 \end{align*}
If in addition $E \subseteq \cl{O}$ is $(d-1)$-regular, define $\X_E^{s,p}(O)\coloneqq\{f|_O : f \in \X_E^{s,p}(\R^d)\}$ for $s \in (1/p,1+1/p)$ with quotient norm
\begin{align*}
 \|f\|_{\X_E^{s,p}(O)} \coloneqq \inf \big\{\|F\|_{\X^{s,p}} : F \in \X_E^{s,p}(\R^d) \text{ and } F|_O = f \big\}.
\end{align*}
\end{definition}

By construction $|_O: \X^{s,p}(\R^d) \to \X^{s,p}(O)$ is bounded. To let $\X^{s,p}(O)$ inherit non-trivial properties of its whole space analogue, a bounded linear right inverse is needed. If $O$ is $d$-regular, this has been constructed in a beautiful paper of Rychkov~\cite[Thm.~5.1]{Rychkov}.

\begin{proposition}[Rychkov]
\label{prop: Rychkov}
Let $O \subseteq \R^d$ be an open, $d$-regular set. Let $\X$ denote either $\H$ or $\W$. For any $s > 0$ and $p \in (1,\infty)$ there exists a bounded linear extension operator $\Ext : \X^{s,p}(O) \to \X^{s,p}(\R^d)$ that serves as a right inverse for $|_O$. Moreover, if $m \geq 1$ is an integer, then $\Ext$ can be taken the same for all $p \in (1,\infty)$ and all $s \in (0,m)$.
\end{proposition}

\begin{remark}
\label{rem: uniformity in Rychkov}
Indeed, though not stated explicitly in \cite{Rychkov}, the consistency of the extension operator becomes apparent from the construction of the operator $\Lambda$ on \cite[p.~155]{Rychkov}.
\end{remark}

Let us stress that Rychkov's operator is not defined on $\L^p(O)$. But in the low-regularity regime $s<1/p$ we can simply extend $\X^{s,p}(O) \to \X^{s,p}(\R^d)$ by zero as we shall see soon. The following definition goes back to Sickel~\cite{Sickel} and Jawerth--Frazier~\cite{Jawerth-Frazier}.

\begin{definition}
\label{def: Sickel class}
Let $t\in (0,1)$. An open set $O \subseteq \R^d$ belongs to the class $\mathcal{D}^t$ if
\begin{align}
\label{eq: Sickel condition}
 \sup_{x \in \bd O} \sup_{0 < r \leq 1} r^{t-d} \int_{\B(x,r) \setminus \bd O} \dist(y,\bd O)^{-t} \; \d y < \infty.
\end{align}
\end{definition}

The relevant examples for us are as follows. For a proof we refer to Proposition~\ref{prop: examples Dt} in the appendix.

\begin{example}
\label{ex: examples Dt class}
An open set with $(d-1)$-regular boundary is of class $\mathcal{D}^t$ for any $t \in (0,1)$. An open set with porous boundary is of class $\mathcal{D}^t$ for some $t \in (0,1)$.
\end{example}

We cite the following multiplier theorem for characteristic functions~\cite[Thm.~4.4]{Sickel}.

\begin{proposition}[Sickel]
\label{prop: Sickel}
Let $O\subseteq \R^d$ be of class $\mathcal{D}^t$ for some $t \in (0,1)$. Let $p \in (1,\infty)$ and $0 \leq s<t/p$. If $\X$ denotes either $\H$ or $\W$, then pointwise multiplication by $\mathds{1}_O$ is a bounded operator on $\X^{s,p}(\R^d)$. For $t(1/p-1) < s < 0$ the dual operator $\mathds{1}_O \varphi \coloneqq \varphi \circ \mathds{1}_O$ is also bounded on $\X^{s,p}(\R^d)$.
\end{proposition}

\begin{corollary}
\label{cor: zero extension bounded}
Let $O \subseteq \R^d$ be an open set with $(d-1)$-regular boundary. Let $\X$ denote either $\H$ or $\W$. Then the zero extension operator
\begin{align*}
 \Ext_0 : \X^{s,p}(O) \to \X^{s,p}(\R^d), \quad \Ext_0 f(x) \coloneqq \begin{cases} 
                                                               f(x) &(\text{if $x \in O$}) \\
                                                               0 &(\text{if $x \in {}^cO$}) 
                                                              \end{cases}
\end{align*}
is bounded provided $p \in (1,\infty)$ and $s \in [0,1/p)$.
\end{corollary}
\section{First properties of function spaces with partially vanishing trace}
\label{Sec: First properties}

We establish first properties and techniques dealing with the spaces introduced in Section~\ref{Subsec: Not spaces on O}. They will frequently be used in the bulk of the paper.

\subsection{Structural properties}
\label{Subsec: Structur results}

Throughout, $\X$ denotes either $\H$ or $\W$. We begin by showing that incorporating boundary conditions leads to complemented subspaces.

\begin{lemma}
\label{lem: XspE complemented}
Let $E \subseteq \R^d$ be $(d-1)$-regular and let $\Res$ and $\Ext$ be the corresponding Jonsson--Wallin operators. Then $\Pro\coloneqq 1 - \Ext \Res$ is a bounded projection
from $\X^{s,p}(\R^d)$ onto $\X^{s,p}_E(\R^d)$  for any $p\in (1,\infty)$ and $s \in (1/p, 1+1/p)$. In particular, $\X^{s,p}_E(\R^d)$ is a closed subspace of $\X^{s,p}(\R^d)$.
\end{lemma}

\begin{proof}
The operator $\Ext \Res$ is bounded on $\X^{s,p}(\R^d)$ by Proposition~\ref{prop: JW}. Since $\Ext$ is a right inverse for $\Res$, we have $(\Ext \Res)^2 = \Ext \Res$, that is to say, $\Ext \Res$ is a projection with the same nullspace as $\Res$. Now, on the one hand, the nullspace of $\Res$ is $\X_E^{s,p}(\R^d)$ and on the other hand, the nullspace of $\Ext \Res$ equals the range of $\Pro$. The conclusion follows.
\end{proof}

Next, we introduce test functions with \emph{support} $\supp(\cdot)$ away from a given set $E$.

\begin{definition}
\label{def: Def CEinfty}
Given $E \subseteq \R^d$, define 
\begin{align*}
 \C_E^\infty(\R^d) \coloneqq \big\{f \in \C_0^\infty(\R^d) : \dist(\supp(f),E) > 0 \big\}
\end{align*}
and if $O \subseteq \R^d$ is any open set, let $\C_E^\infty(O) \coloneqq \{f|_O : f \in \C_E^\infty(\R^d)\}$.
\end{definition}

\begin{lemma}
\label{lem: testfunctions dense}
Let $O \subseteq \R^d$ be open and $E \subseteq \cl{O}$ be $(d-1)$-regular. For $p \in (1,\infty)$ and $s\in (1/p,1]$ the set $\C_E^\infty(O)$ is dense in $\X^{s,p}_E(O)$.
\end{lemma}

\begin{proof}
Since the restriction $|_O: \X^{s,p}_E(\R^d) \to \X^{s,p}_E(O)$ is bounded and onto, it suffices to treat $O = \R^d$. In this case we shall reduce the claim to the fact that any continuous function $f \in \W^{1,p}(\R^d)$ that  vanishes everywhere on a closed set $F \subseteq \R^d$ can be approximated by $\C_F^\infty(\R^d)$-functions in $\W^{1,p}(\R^d)$-norm. This is easily proved by using that $\W^{1,p}(\R^d)$ is closed under truncation, see \cite[Sec.~9.2]{Adams-Hedberg}.

Let $\Pro: \X^{s,p}(\R^d) \to \X_E^{s,p}(\R^d)$ be the bounded projection provided by Lemma~\ref{lem: XspE complemented}. Since $\C_0^\infty(\R^d)$ is dense in $\X^{s,p}(\R^d)$, it suffices to approximate elements in $\Pro(\C_0^\infty(\R^d))$ by test functions from $\C_E^\infty(\R^d)$. Moreover, it suffices to achieve this for the $\W^{1,p}(\R^d)$-norm, which is stronger than the $\X^{s,p}(\R^d)$-norm for we have $s \leq 1$. Since the projection $\Pro$ in Lemma~\ref{lem: XspE complemented} is the same for all admissible values of $s$ and $p$, we have in particular 
\begin{align*}
 \Pro(\C_0^\infty(\R^d)) \subseteq \Pro((\W^{1,d+1} \cap \W^{1,p})(\R^d)) = (\W^{1,d+1}_E \cap \W^{1,p}_E)(\R^d).
\end{align*}
Sobolev embeddings yield for every function in the right-hand space a continuous representative $f$ that vanishes $\cH^{d-1}$-almost everywhere on $E$. By Ahlfors-regularity the intersection of $E$ with arbitrarily small balls centered in $E$ still has positive $\cH^{d-1}$-measure. Thus every point on $F \coloneqq \cl{E}$ is an accumulation point of zeros of $f$. It follows that $f$ vanishes everywhere on $F$ and the above-mentioned approximation result kicks in.
\end{proof}

By a similar argument we prove the surprising feature that Rychkov's extension operator automatically preserves Dirichlet conditions on $(d-1)$-regular sets. Once again, this comes as a byproduct of consistency of the extension operator and Sobolev embeddings and has nothing to do with the particular construction.

\begin{lemma}
\label{lem: Rychkov preserves boundary conditions}
Let $O \subseteq \R^d$ be an open, $d$-regular set, and let $E \subseteq \cl{O}$ be $(d-1)$-regular. Suppose $p \in (1,\infty)$ and $s \in (1/p, 1 +1/p)$. If $\Ext: \X^{s,p}(O) \to \X^{s,p}(\R^d)$ is the extension operator of Proposition~\ref{prop: Rychkov} constructed with $m \geq 2$, then
\begin{align*}
 \Ext : \X^{s,p}_E(O) \to \X^{s,p}_E(\R^d)
\end{align*}
is bounded for the $\X^{s,p}(O) \to \X^{s,p}(\R^d)$-norm. In particular, $\X^{s,p}_E(O)$ is a closed subspace of $\X^{s,p}(O)$.
\end{lemma}

\begin{proof}
By definition of the quotient norm we obtain $\X^{s,p}_E(O) \subseteq \X^{s,p}(O)$ with continuous inclusion of Banach spaces from the fact that $\X^{s,p}_E(\R^d)$ is a closed subspace of $\X^{s,p}(\R^d)$.

We begin with the case $s \leq 1$. Since $\Ext: \X^{s,p}(O) \to \X^{s,p}(\R^d)$ is bounded, it suffices to check that $\Ext$ maps a dense subset of $\X^{s,p}_E(O)$ into $\X^{s,p}_E(\R^d)$. Owing to Lemma~\ref{lem: testfunctions dense} we can take this subset to be $\C_E^\infty(O) = \C_E^\infty(\R^d)|_O$. So, let $f \in \C_E^\infty(\R^d)$. Since $\Ext$ acts consistently, we obtain $\Ext (f|_O) \in (\W^{1,d+1} \cap \X^{s,p})(\R^d)$. Due to Sobolev embeddings $\Ext (f|_O)$ admits a continuous representative and we need to check that it vanishes everywhere on $E$. To this end, we let $B \subseteq \R^d$ be an arbitrary open ball centered in $E \subseteq \cl{O}$ with radius $\r(B) < \dist(\supp(f),E)$. Since $O$ is $d$-regular, $B \cap O$ has positive Lebesgue measure but on this set we have $\Ext (f|_O) = f = 0$ almost everywhere. The conclusion follows.

If $s \in (1, 1+1/p)$ and $f \in \X^{s,p}_E(O)$, then we can use Proposition~\ref{prop: Rychkov} to get $\Ext f \in \X^{s,p}(\R^d)$ and from the inclusion $\X^{s,p}_E(O) \subseteq \X^{1,p}_E(O)$ and the first part of the proof we get $\Ext f \in \X^{1,p}_E(\R^d)$. According to Definition~\ref{def: XspE} this implies $\Ext f \in \X^{s,p}_E(\R^d)$.

As for the final statement, given $f \in \X^{s,p}_E(O)$ we have already seen $\|f\|_{\X^{s,p}_E(O)} \geq \|f\|_{\X^{s,p}(O)}$ and we have just proved $\|f\|_{\X^{s,p}_E(O)} \leq \|\Ext f\|_{\X^{s,p}_E(\R^d)} \lesssim \|f\|_{\X^{s,p}(O)}$.
\end{proof}

\subsection{Symmetric interpolation results}
\label{Subsec: Symmetric interpolation}

We establish symmetric interpolation results for the spaces $\X^{s,p}(O)$ and $\X^{s,p}_E(O)$. \emph{Symmetric} means that either both or none of the spaces are with vanishing trace on $E$. In particular, we prove Theorem~\ref{thm: main result large s}.

The whole theory relies on the retraction-coretraction principle. Given two Banach spaces $X$ and $Y$, a bounded linear operator $\Res: X \to Y$ is called \emph{retraction} if it has a bounded left-inverse $\Ext: Y \to X$ such that $\Res \Ext = 1$ is the identity on $Y$. In this case $\Ext$ is called the associated \emph{coretraction}. It is instructive to think of $\Res$ as a restriction and $\Ext$ a compatible extension operator. The following is proved in \cite[Sec.~1.2.4]{Triebel}. 

\begin{proposition}
\label{prop: retraction-coretraction}
Let $(X_0,X_1)$ and $(Y_0, Y_1)$ be interpolation couples and $\Res: X_0+X_1 \to Y_0 + Y_1$, $\Ext: Y_0 + Y_1 \to X_0+X_1$ be linear operators such that $\Res: X_j \to Y_j$ is a retraction with associated coretraction $\Ext: Y_j \to X_j$ for $j=0,1$. Let $\langle \cdot \,,\cdot \rangle$ denote either a complex or a real interpolation bracket and put $X = \langle X_1, X_2 \rangle$ and $Y = \langle Y_1, Y_2 \rangle$. Then $\Ext \Res$ restricts to a bounded projection in $X$ and 
\begin{align*}
 \Ext : Y \to \Ext \Res(X)
\end{align*}
is an isomorphism of Banach spaces, where $\Ext \Res(X)$ carries the norm of $X$.
\end{proposition}

\begin{remark}
\label{rem: retraction-coretraction}
Above we may apply $\Res$ to the equality of sets $\Ext(Y) = \Ext \Res(X)$ provided by the invertibility of $\Ext : Y \to \Ext \Res(X)$. This yields $Y = \Res(X)$ as sets along with comparability $\|y\|_Y \approx \|\Ext y\|_X$ for $y \in Y$. In particular, if $\Res(X)$ carries the quotient norm inherited from $X/\Ke(\Res)$, then the inclusion $Y \subseteq \Res(X)$ is continuous and the open mapping theorem yields $Y = \Res(X)$ as Banach spaces with equivalent norms.
\end{remark}

An important special case arises when $\Res = \Pro$ is a projection and $\Ext = 1$ is the identity, compare with \cite[Sec.~1.17.1]{Triebel}.

\begin{corollary}
\label{cor: interpolation complemented}
Let $(X_0,X_1)$ be an interpolation couple and $\Pro$ a bounded projection in $X_0 + X_1$ with range $Z$. Then $(Z \cap X_0, Z \cap X_1)$ is an interpolation couple and if $\langle \cdot \,,\cdot \rangle$ denotes either a complex or a real interpolation bracket, then up to equivalent norms
\begin{align*}
 \langle Z \cap X_0, Z \cap X_1 \rangle = Z \cap \langle X_0, X_1 \rangle.
\end{align*}
\end{corollary}

As a first application, we obtain a result similar to Proposition~\ref{prop: interpolation whole space} for spaces on $d$-regular sets. We repeat the well-known argument since it will be re-used several times.

\begin{proposition}
\label{prop: interpol X on subset}
Let $O \subseteq \R^d$ be open and $d$-regular. Let $p_0,p_1 \in (1,\infty)$, $s_0, s_1 \in (0,\infty)$, and $\theta \in (0,1)$. Let $\X$ denote either $\H$ or $\W$. Up to equivalent norms it follows that
\begin{align}
 \tag{i}\label{RychkovInterpol1} [\X^{s_0,p_0}(O), \X^{s_1,p_1}(O)]_{\theta} &= \X^{s,p}(O), \\[8pt]
 \tag{ii}\label{RychkovInterpol2} (\X^{s_0,p_0}(O), \X^{s_1,p_1}(O))_{\theta,p} &= \W^{s,p}(O),
\end{align}
with the two exceptions that in \eqref{RychkovInterpol1} for $\X = \W$ either all or none of $s_0,s_1,s$ have to be integers and in \eqref{RychkovInterpol2} integer $s$ is only permitted when $s_0 = s_1 (=s)$.
\end{proposition}

\begin{proof}
We apply Proposition~\ref{prop: retraction-coretraction} with $X_j \coloneqq \X^{s_j,p_j}(\R^d)$, $Y_j \coloneqq \X^{s_j,p_j}(O)$, $\Res \coloneqq |_O$ the pointwise restriction, and $\Ext$ Rychkov's extension operator from Proposition~\ref{prop: Rychkov} constructed with an integer $m > \max\{s_0,s_1\}$.

Let us prove \eqref{RychkovInterpol1}. According to Proposition~\ref{prop: interpolation whole space} we have $X \coloneqq [X_0, X_1]_{\theta} = \X^{s,p}(\R^d)$. By definition, $\Res(X) = \X^{s,p}(O)$ carries the quotient norm inherited from $X / \Rg(\Res)$. Hence, Remark~\ref{rem: retraction-coretraction} yields $Y \coloneqq [Y_0,Y_1]_{\theta} = \X^{s,p}(O)$ with equivalent norms. The proof of \eqref{RychkovInterpol2} follows \emph{verbatim} from the identity $(X_0,X_1)_{\theta,p} = \W^{s,p}(\R^d)$ also provided by Proposition~\ref{prop: interpolation whole space}.
\end{proof}

\begin{remark}
\label{rem: interpol X on subset}
Suppose that in addition $O$ has $(d-1)$-regular boundary. In the proof above we could then replace Rychkov's extension operator $\Ext$ with the zero extension operator $\Ext_0$ discussed in Corollary~\ref{cor: zero extension bounded}. Consequently, Proposition~\ref{prop: interpol X on subset} remains valid for parameters $s_j \in [0, 1/p_j)$, which includes the case of Lebesgue spaces. 
\end{remark}

The same technique yields the

\begin{proof}[Proof of Theorem~\ref{thm: main result large s}]
First, we assume $O = \R^d$. Proposition~\ref{prop: interpolation whole space} provides the identities analogous to \eqref{C1} and \eqref{C2} for the spaces without Dirichlet conditions. Hence, the claim follows from Corollary~\ref{cor: interpolation complemented} applied to the projection $\Pro$ provided by Lemma~\ref{lem: XspE complemented}. 

Having established the interpolation identities on $\R^d$, we can now pass to the spaces on $O$ via Proposition~\ref{prop: retraction-coretraction} as in the proof of Proposition~\ref{prop: interpol X on subset}. Indeed, if we take $X_j \coloneqq \X^{s_j, p_j}_E(\R^d)$, $Y_j \coloneqq \X^{s_j,p_j}_E(O) = (X_j)|_O$, $\Res \coloneqq |_O$, and $\Ext$ as Rychkov's extension operator, then the only property that needs to be checked is that $\Ext$ maps $Y_j$ boundedly into $X_j$. But the latter is precisely the statement of Lemma~\ref{lem: Rychkov preserves boundary conditions}.
\end{proof}

\subsection{Gluing interpolation scales}
\label{Subsec: Gluing interpolation scales}

We recall a general interpolation technique due to Wolff~\cite{Wolff}. Here, we cite (with adapted notation) the refined version proved in \cite[Thm.~1\&2]{Wolff-refined}. The statement is visualized in Figure~\ref{fig: dich Wolff} for complex interpolation.

\begin{proposition}[Wolff]
\label{prop: Wolff}
Let $X_0,X_\theta,X_\eta, X_1$ be Banach spaces included in a common linear Hausdorff space. Suppose $\theta, \eta, \lambda, \mu \in (0,1)$ satisfy $\theta = \lambda \eta$ and $\eta = (1-\mu)\theta + \mu$, and let $p_\theta, p_\eta \in [1,\infty]$.
\begin{enumerate}
 \item If $X_\theta = [X_0,X_ \eta]_\lambda$ and $X_\eta = [X_\theta, X_1]_\mu$, then also $X_\theta = [X_0,X_1]_\theta$ and $X_\eta = [X_0, X_1]_\eta$.
 
 \item If $X_\theta = (X_0,X_ \eta)_{\lambda, p_\theta}$ and $X_\eta = (X_\theta, X_1)_{\mu,p_\eta}$, then also $X_\theta = (X_0,X_1)_{\theta, p_\theta}$ and $X_\eta = (X_0, X_1)_{\eta, p_\eta}$. 
\end{enumerate}
All equalities above are in the sense of equal sets with equivalent norms.
\end{proposition}

\begin{figure}[ht]
\centering
\begin{tikzpicture}[scale=0.7]
\coordinate[label=left:$X_0$] (X0) at (0,0);
\coordinate[label=left:$X_\theta$] (Xtheta) at (3,0);
\coordinate[label=right:$X_\eta$] (Xeta) at (5,0);
\coordinate[label=right:$X_1$] (X1) at (9,0);
\coordinate[label=above:${[X_0, X_\eta]_\lambda}$] (I1) at (3,0.7);
\coordinate[label=below:${[X_\theta, X_1]_\mu}$] (I1) at (5,-0.7);

\draw (0,0) -- (0,0.7);
\draw (0,0.7) -- (5,0.7);
\draw (5,0.7) -- (5,0);
\draw[dashed] (3,0) -- (3,0.7);

\draw (3,0) -- (3,-0.7);
\draw (3,-0.7) -- (9,-0.7);
\draw (9,-0.7) -- (9,0);
\draw[dashed] (5,-0.7) -- (5,0);

\fill (X0) circle (3pt);
\fill (Xtheta) circle (3pt);
\fill (Xeta) circle (3pt);
\fill (X1) circle (3pt);
\end{tikzpicture}
\caption{Assuming the interpolation identities indicated by dashed lines, Wolff's result recovers $X_\theta$ and $X_\eta$ as interpolation spaces associated with the couple $(X_0, X_1)$ for the correct convex combination parameters $\theta$ and $\eta$, respectively.}
\label{fig: dich Wolff}
\end{figure}
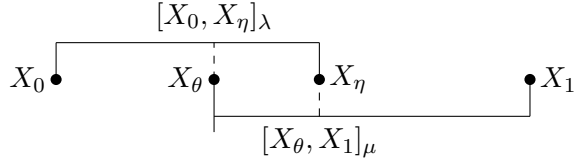

For further reference we demonstrate once in detail how the results of Proposition~\ref{prop: interpol X on subset} and Remark~\ref{rem: interpol X on subset} can be patched together using Wolff's result.

\begin{proposition}
\label{prop: interpol X on subset complete}
If in the setting of Proposition~\ref{prop: interpol X on subset} the boundary $\bd O$ is $(d-1)$-regular, then the conclusion remains valid for $s_0, s_1 \in [0,\infty)$.
\end{proposition}

\begin{proof}
In view of Proposition~\ref{prop: interpol X on subset}, Remark~\ref{rem: interpol X on subset}, and symmetry of the assumption, we only have to treat the case $s_0 = 0$ and $s_1>0$. For any $\eta \in (0,1)$ we abbreviate the relevant convex combinations by $s_\eta \coloneqq \eta s_1$ and $1/p_\eta \coloneqq (1-\eta)/p_0 + \eta/p_1$. 

We begin with \eqref{RychkovInterpol1}. Since $s_0$ is an integer, we are only claiming something new in the case $\X = \H$. We have to prove for all $\eta \in (0,1)$ the equality
\begin{align}
\label{eq1: interpol s large complete}
 [\H^{s_0, p_0}(O), \H^{s_1,p_1}(O)]_\eta = \H^{s_\eta, p_\eta}(O).
\end{align}
Throughout, the reader should keep in mind Figure~\ref{fig: dich Wolff}. Let us first suppose $s_\eta < 1/p_\eta $ so that $\H^{s_\eta,p_\eta}(\R^d)$ belongs to the regime covered by Remark~\ref{rem: interpol X on subset}. We pick $\theta \in (0,\eta)$ and $\lambda, \mu \in (0,1)$ such that $\theta = \lambda \eta$ and $\eta = (1-\mu)\theta + \mu$. The quadruple of spaces $(X_i)_i \coloneqq (\H^{s_i, p_i}(O))_i$ satisfies the assumption in part (i) of Wolff's result owing to Remark~\ref{rem: interpol X on subset} and Proposition~\ref{prop: interpol X on subset}. Hence, we obtain \eqref{eq1: interpol s large complete}. Now, suppose $s_\eta \geq 1/p_\eta$. Due to $s_0 = 0$ we can pick $\theta \in (0,\eta)$ to arrange $s_\theta  < 1/p_\theta$. The first part of the proof applies to $s_\eta$ in place of $s_1$ and yields $[X_0,X_\eta]_\lambda = X_\theta$. Consequently, we can apply Wolff's result with the same numerology as before to obtain~\eqref{eq1: interpol s large complete}.

As for \eqref{RychkovInterpol2}, the claim for $\W$-spaces follows \emph{verbatim} by using part (ii) of Wolff's result with $p_\theta, p_\eta$ corresponding to $\theta,\eta$ as above and systematically replacing $\H$ by $\W$. 

Real interpolation of $\H$-spaces requires a different argument since the result will be a $\W$-space. We rely on the one-sided reiteration theorem in Proposition~\ref{prop: 1-sided reiteration} below. Indeed, given $\theta \in (0,1)$ we pick $\eta \in (0,\theta)$ and write $\theta = (1-\lambda)\eta + \lambda$ with $\lambda \in (0,1)$. Then we use in succession one-sided reiteration, complex interpolation of $\H$-spaces established above, and Proposition~\ref{prop: interpol X on subset}, to give
\begin{align*}
 \big(\H^{s_0,p_0}(O), \H^{s_1,p_1}(O) \big)_{\theta,p_\theta} 
 &= \big([\H^{s_0,p_0}(O), \H^{s_1,p_1}(O)]_\eta, \H^{s_1,p_1}(O) \big)_{\lambda, p_\theta} \\
 &= \big(\H^{s_\eta,p_\eta}(O), \H^{s_1,p_1}(O) \big)_{\lambda, p_\theta}\\
 &= \W^{s_\theta, p_\theta}(O).
\end{align*}
Concerning the last line we remark that $(1-\lambda)s_\eta + \lambda s_1 = s_\theta$ and $(1-\lambda)/p_\eta + \lambda/p_1 = 1/p_\theta$ hold by construction.
\end{proof}

The reiteration result that we have invoked above is as follows. We refer to \cite[Sec.~1.10.3, Thm.~2]{Triebel} for real interpolation and to \cite{Cwikel-Reiteration} for complex interpolation, noting that in the latter case the density of $X_0 \cap X_1$ in $X_1$ guarantees $[X_0,X_1]_1 = X_1$.

\begin{proposition}
\label{prop: 1-sided reiteration}
Let $(X_0, X_1)$ be an interpolation couple. Let $\eta, \lambda \in (0,1)$ and put $\theta = (1-\lambda)\eta + \lambda$. The interpolation identity
\begin{align*}
 \langle [X_0,X_1]_\eta, X_1 \rangle_\lambda = \langle X_0, X_1 \rangle_{\theta}
\end{align*}
holds up to equivalent norms in the following cases. If $\langle \cdot \,,\cdot \rangle$ is a $(\cdot\,,p)$-real interpolation bracket with $p \in [1,\infty]$ fixed or if $\langle \cdot \,,\cdot \rangle$ is the complex interpolation bracket and $X_0 \cap X_1$ is dense in $X_1$. 
\end{proposition}

\subsection{Non-symmetric interpolation: The easy inclusion}
\label{Subsec: Easy inclusion}

The main difficulty in Theorem~\ref{thm: main result} lies in proving the inclusion ``$\supseteq$''. Indeed, here we can already prove

\begin{proposition}
\label{prop: easy inclusion}
Let $O \subseteq \R^d$ be an open, $d$-regular set with $(d-1)$-regular boundary, and let $D \subseteq \cl{O}$ be $(d-1)$-regular. Let $p_0,p_1 \in (1,\infty)$, $s_0 \in [0, 1/p_0)$, $s_1 \in (1/p_1,1]$, and $\theta \in (0,1)$. Define $p$ and $s$ as in \eqref{eq: interpolating paramaters}.
Then there are continuous inclusions
\begin{align}
\tag{i}\label{Easy1}[\X^{s_0,p_0}(O), \X_D^{s_1,p_1}(O)]_{\theta} &\subseteq \begin{cases} 
                                                                 \X_D^{s,p}(O) &(\text{if $s>1/p$})\\ \X^{s,p}(O)  &(\text{if $s<1/p$})
                                                                \end{cases}
\intertext{and}
\tag{ii}\label{Easy2}\noeqref{Easy2}
(\X^{s_0,p_0}(O), \X_D^{s_1,p_1}(O))_{\theta,p} &\subseteq \begin{cases} 
                                                                 \W_D^{s,p}(O) &(\text{if $s>1/p$})\\ \W^{s,p}(O)  &(\text{if $s<1/p$})
                                                                \end{cases}                                                      
\end{align}
with the exception that $s_0 \neq 0$ and $s_1 \neq 1$ are required in \eqref{Easy1} for $\X = \W$.
If $p_0 = p_1$, then the result remains true for all $s_1 \in (1/p_1,1+1/p_1)$ with the additional exception that only in \eqref{Easy1} for $\X = \H$ the value $s = 1$ is permitted.
\end{proposition}

\begin{proof}
First, we check that Proposition~\ref{prop: 1-sided reiteration} applies in its real and its complex version to the couple $(\X^{s_0,p_0}(O),\X^{s_1,p_1}(O))$. If $s_1 \leq 1$ then $\X^{s_0,p_0}(O) \cap \X_D^{s_1,p_1}(O) \supseteq \C_D^\infty(O)$ is dense in $\X_D^{s_1,p_1}(O)$ by Lemma~\ref{lem: testfunctions dense} and if $p_0 = p_1$ then $\X^{s_0,p_0}(O) \cap \X_D^{s_1,p_1}(O) = \X_D^{s_1,p_1}(O)$ for all $s_1 \in (1/p_1,1+1/p_1)$. This being said, we denote by $\langle \cdot \,,\cdot \rangle$ either the $(\cdot\,,p)$-real or the complex interpolation bracket and treat all assertions simultaneously. 
 
By definition we have $\X^{s_1,p_1}_D(O) \subseteq \X^{s_1,p_1}(O)$ and hence we get
\begin{align}
\label{eq1: easy inclusion}
 \langle \X^{s_0,p_0}(O),\X^{s_1,p_1}_D(O) \rangle_\theta \subseteq \langle \X^{s_0,p_0}(O),\X^{s_1,p_1}(O) \rangle_\theta
\end{align}
with continuous inclusion. The interpolation space on the right has been determined in Proposition~\ref{prop: interpol X on subset complete}. It coincides (up to equivalent norms) with $\W^{s,p}(O)$ in case of real interpolation and with $\X^{s,p}(O)$ in case of complex interpolation. In the case $s<1/p$ this already is the desired conclusion. 
 
Let now $s>1/p$. We fix $\eta \in (0,\theta)$ sufficiently close to $\theta$, so to arrange $1/p_\eta \coloneqq (1-\eta)/p_0 + \eta/p_1$ and $s_\eta \coloneqq (1-\eta)s_0 + \eta s_1$ satisfying $s_\eta > 1/p_\eta $. We write $\theta = (1-\lambda)\eta + \lambda$ with $\lambda \in (0,1)$. From Proposition~\ref{prop: 1-sided reiteration} and the reasoning in the first case we obtain 
\begin{align*}
 \langle \X^{s_0,p_0}(O),\X^{s_1,p_1}_D(O) \rangle_\theta 
&= \langle [\X^{s_0,p_0}(O),\X^{s_1,p_1}_D(O)]_\eta, \X^{s_1,p_1}_D(O) \rangle_\lambda \\
&\subseteq \langle \X^{s_\eta,p_\eta}(O), \X^{s_1,p_1}_D(O) \rangle_\lambda
\end{align*}
with continuous inclusion. Let $\Ext$ be Rychkov's extension operator for $O$. From Lemma~\ref{lem: Rychkov preserves boundary conditions} and the above we can infer by interpolation that
\begin{align}
\label{eq2: easy inclusion}
 \Ext: \langle \X^{s_0,p_0}(O),\X^{s_1,p_1}_D(O) \rangle_\theta \to \langle \X^{s_\eta,p_\eta}(\R^d),\X^{s_1,p_1}_D(\R^d) \rangle_\lambda \eqqcolon Y
\end{align}
is bounded. As before, we see that $Y$ is continuously included into $\W^{s,p}(\R^d)$ in case of real interpolation and into $\X^{s,p}(\R^d)$ in case of complex interpolation. 

Consider the Jonsson--Wallin restriction operator to $D$, see Proposition~\ref{prop: JW}. It maps $\X^{s_\eta, p_\eta}(\R^d)$ boundedly into $\W^{s_\eta-1/p_\eta,p_\eta}(D)$ since we have $s_\eta  > 1/p_\eta$ and it maps $\X^{s_1,p_1}_D(\R^d)$ into $\{0\}$ by definition. By interpolation it maps $Y$ into $\langle \W^{s_\eta-1/p_\eta,p_\eta}(D), \{0\} \rangle_\lambda$. This interpolation space equals $\{0\}$ since it contains $\{0\}$ as a dense subspace. Hence, we have continuous inclusion of $Y$ into $\W^{s,p}_D(\R^d)$ in case of real interpolation and into $\X^{s,p}_D(\R^d)$ in case of complex interpolation. By \eqref{eq2: easy inclusion} every function in $\langle \X^{s_0,p_0}(O),\X^{s_1,p_1}_D(O) \rangle_\theta$ has an extension in $Y$ in virtue of a bounded extension operator. The required continuous inclusion follows.
\end{proof}
\section{Proof of Theorem~\ref{thm: main result}}
\label{Sec: Complex}

It will be convenient to reformulate the assumptions of Theorem~\ref{thm: main result} as follows, relying on Examples~\ref{ex: d-set} and \ref{ex: l-set}.

\begin{assumption}
\label{ass: complex}
The set $O\subseteq \R^d$ is open, $d$-regular, has $d$-regular complement, and $(d-1)$-regular boundary. The Dirichlet part $D \subseteq \bd O$ is $(d-1)$-regular and $\bd D$ is porous in $\bd O$. Moreover, $O$ satisfies a uniform Lipschitz condition around $\overline{\bd O \setminus D}$.
\end{assumption}

Throughout the whole section let $\X$ denote either $\H$ or $\W$. We are given $p_0, p_1 \in (1,\infty)$, $s_0 \in [0,1/p_0)$, $s_1 \in (1/p_1,1]$, and $\theta \in (0,1)$. When concerned with complex interpolation for $\X = \W$, we implictly restrict ourselves to $s_0 \neq 0$ and $s_1 \neq 1$. Our goal is to establish set inclusions
\begin{align} \label{Ip complex}
    [\X^{s_0,p_0}(O), \X^{s_1,p_1}_D(O)]_\theta \supseteq \begin{cases}
                                       \X^{s,p}_D(O) &(\text{if $s > 1/p$}) \\
                                       \X^{s,p}(O) &(\text{if $s < 1/p$})
                                      \end{cases}
\end{align}
and
\begin{align} \label{Ip real}
    (\X^{s_0,p_0}(O), \X^{s_1,p_1}_D(O))_{\theta,p} \supseteq \begin{cases}
                                       \W^{s,p}_D(O) &(\text{if $s > 1/p$}) \\
                                       \W^{s,p}(O) &(\text{if $s < 1/p$})
                                      \end{cases}.
\end{align}
This will complete the proof of Theorem~\ref{thm: main result} since under Assumption~\ref{ass: complex} the converse inclusions are continuous due to Proposition~\ref{prop: easy inclusion} and hence become equalities with equivalent norms thanks to the bounded inverse theorem.

\subsection{Road map to the proof}
\label{Subsec: Roadmap complex}

We give the outline for complex interpolation. The real case will be treated in the same way up to replacing the complex interpolation bracket with the $(\cdot\,,p)$-real interpolation bracket and keeping in mind that real interpolation spaces of $\X$-spaces are always $\W$-spaces.

First, we show in Section \ref{Subsec: Pure Dirichlet complex} that \eqref{Ip complex} and \eqref{Ip real} hold in the case $D=\bd O$ of pure Dirichlet boundary condition. Then the inclusion with general $D$ and $s\in (0,1/p)$ follows readily:
\begin{align} \label{Eqn: small s from Dirichlet}
    \X^{s,p}(O) \subseteq 
    \bigl[\X^{s_0,p_0}(O),\X^{s_1,p_1}_{\bd O}(O)\bigr]_\theta 
    \subseteq \bigl[\X^{s_0,p_0}(O),\X^{s_1,p_1}_D(O)\bigr]_\theta.
\end{align}
In the case $s \in (1/p,1)$ we localize in order to reduce the problem to pure Dirichlet interpolation and interpolation with mixed boundary conditions, but for a simpler geometry. Precisely, we will have $O = \R^d_+$ the upper half-space and $E_i$ a transformed version of a portion of $D$ with a security area for good measure that is still $(d-1)$-regular and has porous boundary in $\bd \R^d_+ \cong \R^{d-1}$. Then we have to show that
\begin{align} \label{Eqn: half-space inclusion}
    \X^{s,p}_{E_i}(\R^d_+) \subseteq \bigl[\X^{s_0,p_0}(\R^d_+),\X^{s_1,p_1}_{E_i}(\R^d_+)\bigr]_\theta.
\end{align}
This reduction will be done in Section~\ref{Subsec: Localization}. 

The heart of the matter lies in showing \eqref{Eqn: half-space inclusion} in Section~\ref{Subsec: conclusion}. To do so, we decompose $f\in \X^{s,p}_{E_i}(\R^d_+)$ as $f=(f-\Ext\Res f)+\Ext\Res f$, where $\Res$ is the restriction to $\bd \R^d_+$ and $\Ext$ is a corresponding extension operator. The term $f-\Ext\Res f$ will be in $[\X^{s_0,p_0}(\R^d_+),\X^{s_1,p_1}_{E_i}(\R^d_+)]_\theta$ because it satisfies pure Dirichlet boundary conditions on $\bd \R^d_+$. The argument for $\Ext\Res f$ happens completely at the boundary and is displayed in Figure~\ref{fig: road map complex}.

\begin{figure}[ht]
\centering
\[ \begin{tikzcd}
        \X^{s,p}_{E_i}(\R^d_+) \ar{dd}{\Res} & \left[\X^{s_0,p_0}(\R^d_+),\X^{s_1,p_1}_{E_i}(\R^d_+)\right]_\theta \\
                        & \left[\X^{1/q-\varepsilon,q}(\R^d_+),\X^{s_1,p_1}_{E_i}(\R^d_+)\right]_\eta \ar[equal]{u}{\text{reiteration}} \\
        \W^{s-1/p,p}_\bullet({}^c E_i) \ar[equal]{r}{(\heartsuit)} & \left[\W^{-\varepsilon,q}_\bullet({}^c E_i),\W^{s_1-1/p_1,p_1}_\bullet({}^c E_i)\right]_\eta \ar{u}{\Ext}
\end{tikzcd}\]
\caption{Schematic presentation of the main argument to prove the inclusion ``$\supseteq$'' in part \eqref{C3} of Theorem~\ref{thm: main result}.}
\label{fig: road map complex}
\end{figure}

Here, $\W^{s,p}_\bullet({}^c E_i)$ is a subspace of $\W^{s,p}(\R^{d-1})$ with zero condition on the full dimensional set $E_i \subseteq \R^{d-1}$ and $q$, $\varepsilon$ and $\eta$ are parameters yet to be determined. We need to establish
\begin{itemize}
    \item the construction of an extension operator $\Ext$ from $\bd \R^d_+$ to $\R^d_+$ which is consistent in $s\in \R \setminus \IZ$ and $p\in (1,\infty)$ and
    \item the precise definition of the spaces $\W^{s,p}_\bullet({}^cE_i)$ for a suitable range of $s$ including verification of the interpolation identity $(\heartsuit)$.
\end{itemize}

The passage through spaces of negative order in $(\heartsuit)$ is inevitable and can be implemented in virtue of Proposition~\ref{prop: Sickel} only because $\bd D$ is porous in $\bd O$.

\subsection{Spaces of functions vanishing on a full-dimensional subset}
\label{Subsec: Bullet spaces}

For this part we work with a $d$-regular set $U \subseteq \R^d$ whose boundary is a Lebesgue null set and whose interior $\mathring{U}$ is of class $\cD^t$ for some $t\in (0,1)$, compare with Definition~\ref{def: Sickel class}. 

We remark that most results stated in Section~\ref{Subsec: Not spaces on O} for open sets still apply in this context. Pointwise multiplication by the characteristic functions of $U$ and $\mathring{U}$ coincide on $\L^p(\R^d)$. Moreover, $\mathring{U}$ is $d$-regular and the corresponding Rychkov's extension operators can also be regarded as extension operators for functions defined on $U$.

Let $\Res$ denote the pointwise restriction operator $|_U$ and let $\Ext$ denote \emph{some} extension operator $\X^{s,p}(U)\to \X^{s,p}(\R^d)$. We will specify consistency requirements later on. For $p\in (1,\infty)$ and $s\in (0,\infty)$ we define
\begin{align}
    \X^{s,p}_\bullet({}^cU) \coloneqq \{ f\in \X^{s,p}(\R^d): \Res f = 0 \}
\end{align}
with subspace topology. This subspace is complemented in virtue of the projection $1- \Ext \Res$. 

The pointwise multiplier $\mathds{1}_U$ is bounded on $\X^{s,p}(\R^d)$ for $t(1/p-1) < s < t/p$ due to Proposition~\ref{prop: Sickel}. This allows us to extend the definition of $\X^{s,p}_\bullet({}^cU)$ to such $s$ by
\begin{align}
    \X^{s,p}_\bullet({}^cU) \coloneqq \{ f-\mathds{1}_U f: f\in \X^{s,p}(\R^d) \},
\end{align}
where the topology is again the subspace topology. Note that for $s\in (0,t/p)$ this gives the same space as before and that now $1-\mathds{1}_U$ becomes the complementing projection.

The following lemma captures the interpolation behavior of these spaces.

\begin{lemma} \label{lem: Bullet interpolation}
    Let $p_0,p_1 \in (1,\infty)$, $s_0\in (t(1/p_0-1),\infty)$, $s_1\in (t(1/p_1-1),\infty)$, and $\theta \in (0,1)$. Up to equivalent norms it follows that
    \begin{align}
       \tag{i}\label{Bullet1}[\X^{s_0,p_0}_\bullet({}^cU), \X^{s_1,p_1}_\bullet({}^cU)]_{\theta} &= \X^{s,p}_\bullet({}^cU), \\[8pt]
       \tag{ii}\label{Bullet2}(\X^{s_0,p_0}_\bullet({}^cU), \X^{s_1,p_1}_\bullet({}^cU))_{\theta,p} &= \W^{s,p}_\bullet({}^cU),
    \end{align}
    with the two exceptions that in \eqref{Bullet1} for $\X = \W$ either all or none of $s_0,s_1,s$ have to be integers and that in \eqref{Bullet2} integer $s$ is only permitted when $s_0 = s_1 (=s)$.
\end{lemma}

\begin{proof}
    By symmetry we may assume $s_0\leq s_1$. In virtue of Corollary~\ref{cor: interpolation complemented} we shall transfer the interpolation identities of Proposition~\ref{prop: interpolation whole space} for the $\X^{s,p}(\R^d)$-spaces to the $\X^{s,p}_\bullet({}^c U)$-spaces. We only have to identify suitable projections $\Pro$. 
    
    If $s_0 > 0$, then we pick a Rychkov's extension operator $\Ext$ that is consistent up to a positive integer greater $s_1$ and use $\Pro \coloneqq 1-\Ext\Res$. 
    
    Now assume $s_0\leq 0$. If $s_1 < t/p_1$, then we can directly use $\Pro \coloneqq 1-\mathds{1}_U$. Otherwise, there are $p\in (1,\infty)$ and $s \in (0,t/p)$ such that $(s,1/p)^\top$ lies on the segment connecting $(s_0,1/p_0)^\top$ and $(s_1,1/p_1)^\top$ in the $(s,1/p)$-plane. If necessary, we can arrange that $s$ is not an integer. We have just obtained interpolation for the spaces on the segment connecting $(s_0,1/p_0)^\top$ and $(s,1/p)^\top$ and in order to conclude, we patch this interpolation scale together with the one for positive differentiability by the technique illustrated in the proof of Proposition~\ref{prop: interpol X on subset complete}.
\end{proof}

\subsection{The case of pure Dirichlet conditions}
\label{Subsec: Pure Dirichlet complex}

For this part we strengthen our previous requirements on $U \subseteq \R^d$ to the effect that it should be a closed $d$-regular set with $(d-1)$-regular boundary. 

It follows that $\bd U$ is a Lebesgue null set and we claim that $\mathring{U}$ is of class $\cD^t$ for all $t\in (0,1)$. Indeed, by Example~\ref{ex: examples Dt class} the open set ${}^c U$ has this property and since we have $\bd \mathring{U} \subseteq \bd U = \bd \,({}^c U)$ with set difference of zero Lebesgue measure, we see by the very definition that if ${}^c U$ is of class $\cD^t$, then so is $\mathring{U}$.

We start out with a reformulation of Corollary~\ref{cor: zero extension bounded}.

\begin{lemma} \label{lem: Restriction HspBullet small s}
    If $p \in (1,\infty)$ and $s\in [0,1/p)$, then $\X^{s,p}({{}^cU})=\X^{s,p}_\bullet({{}^cU})|_{{}^cU}$ with equivalent norms.
\end{lemma}

\begin{proof}
    The inclusion $\X^{s,p}_\bullet({{}^cU})|_{{}^cU} \subseteq \X^{s,p}({{}^cU})$ is clear. For the converse let $f\in \X^{s,p}({{}^cU})$ and $F$ an extension of $f$ in $\X^{s,p}(\R^d)$. We get $\mathds{1}_{{}^cU} \, F \in \X^{s,p}(\R^d)$ owing to Corollary \ref{cor: zero extension bounded}. Hence, we have $\mathds{1}_{{}^cU} \,F \in \X^{s,p}_\bullet({{}^cU})$ and $f=(\mathds{1}_{{}^cU}\, F)|_{{}^cU} \in \X^{s,p}_\bullet({{}^cU})|_{{}^cU}$ follows. For the boundedness, we calculate
    \begin{align*}
     \|f\|_{\X^{s,p}_\bullet({{}^cU})|_{{}^cU}} \leq \|\mathds{1}_{{}^cU}\, F\|_{\X^{s,p}(\R^d)} \lesssim \|F\|_{\X^{s,p}(\R^d)}
    \end{align*}
    and take the infimum over all such extensions $F$.
\end{proof}

In order to proceed, we need a generic re-norming lemma and its consequence for the pointwise multiplication by $\mathds{1}_{{}^cU}$. To fix ideas for the following, we decided to include a proof even though the result is known in the literature.

\begin{lemma} \label{lem: Xsp norm with gradient}
    If $p \in (1,\infty)$ and $s \in \R$, then
    \begin{align}
    \label{eq1: Xsp norm with gradient}
    \| f \|_{\X^{s,p}} \approx \|f\|_{\X^{s-1,p}} + \| \nabla f \|_{\X^{s-1,p}} \qquad (f \in \S'(\R^d)).
    \end{align}
\end{lemma}

\begin{proof}
    In the following all function spaces are on $\R^d$ and we omit the dependence. The operator $\cI_{-1} f \coloneqq \cF^{-1}(1+|\xi|^2)^{1/2} \cF f$ is invertible from $\S'$ into itself.  By definition it restricts to an isomorphism $\cI_{-1}: \H^{s,p} \to \H^{s-1,p}$. By interpolation the same holds for $\cI_{-1}: \W^{s,p} \to \W^{s-1,p}$, see Proposition~\ref{prop: interpolation whole space}. Hence, we find for all $f \in \S'$, 
    \begin{align*}
     \|f\|_{\X^{s,p}} \approx \|\cF^{-1}(1+|\xi|^2)^{1/2} \cF f\|_{\X^{s-1,p}}.
    \end{align*}
    Comparing with \eqref{eq1: Xsp norm with gradient}, we see that it remains to prove
    \begin{align}
    \label{eq2: Xsp norm with gradient}
     \|f\|_{\X^{s-1,p}} + \sum_{j=1}^d \| \cF^{-1} \xi_j \cF f \|_{\X^{s-1,p}} \approx \|\cF^{-1}(1+|\xi|^2)^{1/2} \cF f\|_{\X^{s-1,p}}.
    \end{align}
    To this end we consider Fourier multipliers $f \mapsto \cF^{-1}m\cF f$, defined on $\S'$ via a smooth and bounded function $m: \R^d \to \IC$, to pass from one side to the other. If such multiplier is bounded on $\L^p$, then it is bounded on $\H^{k,p}$ for all integers $k$ since it commutes with $\cI_{-1}$ and its inverse.  Hence, it is bounded on $\X^{s,p}$ for all $s \in \R$ by interpolation. 
    This being said, we obtain ``$\lesssim$'' in \eqref{eq2: Xsp norm with gradient} by considering the Fourier multipliers associated with $(1+|\xi|^2)^{-1/2}$ and $\xi_j (1+|\xi|^2)^{-1/2}$. Their $\L^p$-boundedness follows easily from the Mihlin multiplier theorem~\cite[Thm.~6.1.6]{BL}. Next, we pick a smooth function $\chi: \R \to [0,1]$ that vanishes on $(-1,1)$ and is identically $1$ outside of $[-2,2]$ in order to write
    \begin{align*}
     (1+|\xi|^2)^{1/2} = \bigg(\frac{(1+|\xi|^2)^{1/2}}{1+\sum_{j=1}^d \chi(\xi_j) |\xi_j|} \bigg) + \sum_{j=1}^d \bigg(\frac{(1+|\xi|^2)^{1/2}}{1+\sum_{j=1}^d \chi(\xi_j) |\xi_j|} \bigg) \bigg(\frac{\chi(\xi_j)|\xi_j|}{\xi_j} \bigg) \xi_j.
    \end{align*}
    Again by Mihlin's theorem each bracket corresponds to an $\L^p$-bounded Fourier multiplier. This yields the converse estimate ``$\gtrsim$''.
\end{proof}

\begin{lemma} \label{lem: multiplier large s}
    For $p \in (1,\infty)$ and $s\in (1/p,1+1/p)$ pointwise multiplication by $\mathds{1}_{{}^cU}$ is $\X^{s,p}_{\bd U}(\R^d) \to \X^{s,p}_\bullet({{}^cU})$-bounded.
\end{lemma}

\begin{proof}
    For $f\in \C_{\bd U}^\infty(\R^d)$ we have that $\nabla(\mathds{1}_{{}^cU} \, f)=\mathds{1}_{{}^cU} \, \nabla f$. Hence, we can combine Lemma~\ref{lem: Xsp norm with gradient} and Proposition~\ref{prop: Sickel} to the effect that 
    \begin{align}
    \label{eq1: multiplier large s}
    \begin{split}
        \|\mathds{1}_{{}^cU} \, f\|_{\X^{s,p}} 
                                         &\approx \|\mathds{1}_{{}^cU} \, f\|_{\X^{s-1,p}} + \|\mathds{1}_{{}^cU} \, \nabla f\|_{\X^{s-1,p}} \\
                                         &\lesssim \|f\|_{\X^{s-1,p}} + \|\nabla f\|_{\X^{s-1,p}} \\
                                         &\approx \|f\|_{\X^{s,p}}.
    \end{split}
    \end{align}
    For $s \in (1/p,1]$ we can use that $\C_{\bd U}^\infty(\R^d)$ is dense in $\X^{s,p}_{\bd U}(\R^d)$ by Lemma~\ref{lem: testfunctions dense} to conclude that $\mathds{1}_{{}^cU}: \X^{s,p}_{\bd U}(\R^d) \to \X^{s,p}(\R^d)$ is bounded. That it actually maps into the closed subspace $\X^{s,p}_\bullet({{}^cU})$ follows by construction. Suppose now $s \in (1, 1+1/p)$. The commutation $\nabla(\mathds{1}_{{}^cU} \, \cdot)=\mathds{1}_{{}^cU} \nabla(\cdot)$ extends by density to all $f \in \X^{1,p}_{\bd U}(\R^d)$. Hence, it holds in particular on $\X^{s,p}_{\bd U}(\R^d)$ and the calculation \eqref{eq1: multiplier large s} re-applies.
\end{proof}

We get the analogue of Lemma~\ref{lem: Restriction HspBullet small s} in the case $s\in (1/p,1+1/p)$.

\begin{lemma} \label{lem: Restriction HspBullet large s}
    If $p \in (1,\infty)$ and $s\in (1/p,1+1/p)$, then $\X^{s,p}_{\bd U}({{}^cU})=\X^{s,p}_\bullet({{}^cU})|_{{}^cU}$ with equivalent norms.
\end{lemma}

\begin{proof}
    The inclusion $\X^{s,p}_{\bd U}({{}^cU}) \subseteq \X^{s,p}_\bullet({{}^cU})|_{{}^cU}$ works as in the proof of Lemma~\ref{lem: Restriction HspBullet small s} when using Lemma~\ref{lem: multiplier large s} instead of Corollary \ref{cor: zero extension bounded}.
    
    Conversely, let $f \in \X^{s,p}_\bullet({{}^cU})$. Since $f$ is in particular a member of $\X^{s,p}(\R^d)$, we find a sequence $(f_n)_{n} \subseteq \C^\infty_0(\R^d)$ that approximates $f$ in the topology of $\X^{s,p}(\R^d)$. Let $\Ext$ be Rychkov's extension operator for $U$, which, as we recall, acts consistently on $\W^{1,d+1}(\R^d)$. We apply the projection $\Pro=\Id-\Ext\Res$ to that sequence. Since $\Pro$ projects onto $\X^{s,p}_\bullet({{}^cU})$, we get $\Pro f_n = 0$ almost everywhere on $U$ on the one hand and $\Pro f_n \in \C(\R^d)$ by Sobolev embeddings on the other hand. By $d$-regularity, the intersection of $U$ with balls centered in $U$ has positive Lebesgue measure. Hence, the $\Pro f_n$ vanish everywhere on $U$. In particular they vanish on $\bd U$, which means $\Pro f_n \in \X^{s,p}_{\bd U}(\R^d)$. Now, since $f_n\to f$ in $\X^{s,p}(\R^d)$, also $\Pro f_n \to \Pro f = f$ in $\X^{s,p}(\R^d)$, which gives $f\in \X^{s,p}_{\bd U}(\R^d)$.
\end{proof}

Eventually, we can transfer the interpolation settled in Lemma~\ref{lem: Bullet interpolation} to the spaces incorporating pure Dirichlet boundary conditions. Since we can take $U = {}^c O$, this gives the full claim of Theorem~\ref{thm: main result} for pure Dirichlet conditions.

\begin{proposition} \label{prop: pure Dirichlet interpolation}
    Let $p_0,p_1 \in (1,\infty)$, $s_0 \in [0,1/p_0)$, $s_1 \in (1/p_1,1]$, and $\theta \in (0,1)$. There are continuous inclusions
    \begin{align}
        \tag{i}\label{D1} [\X^{s_0,p_0}({{}^cU}), \X_{\bd U}^{s_1,p_1}({{}^cU})]_{\theta} &\supseteq \begin{cases} 
                            \X_{\bd U}^{s,p}({{}^cU}) &(\text{if $s>1/p$})\\ 
                            \X^{s,p}({{}^cU})  &(\text{if $s<1/p$})
                                                                    \end{cases},
                                                                    \\[8pt]
        \tag{ii}\label{D2}\noeqref{D2} (\X^{s_0,p_0}({{}^cU}), \X_{\bd U}^{s_1,p_1}({{}^cU}))_{\theta,p} &\supseteq \begin{cases} 
                                                                        \W_{\bd U}^{s,p}({{}^cU}) &(\text{if $s>1/p$})\\ 
                                                                        \W^{s,p}({{}^cU})  &(\text{if $s<1/p$})
                                                                     \end{cases},
    \end{align}
    with the exception that $s_0 \neq 0$ and $s_1 \neq 1$ are required in \eqref{D1} for $\X = \W$. If in addition ${}^cU$ is $d$-regular, then both inclusions become equalities with equivalent norms.
\end{proposition}

\begin{proof}  
    Let $\langle\,\cdot\,,\,\cdot\,\rangle$ denote either the $\theta$-complex or $(\theta,p)$-real interpolation bracket. Using Lemma~\ref{lem: Restriction HspBullet small s} and Lemma~\ref{lem: Restriction HspBullet large s}, we get
    \begin{align}
    \begin{split}
        \langle \X^{s_0,p_0}_\bullet({{}^cU}), \X^{s_1,p_1}_\bullet({{}^cU}) \rangle |_{{}^cU}
        &\subseteq \langle \X^{s_0,p_0}_\bullet({{}^cU})|_{{}^cU}, \X^{s_1,p_1}_\bullet({{}^cU})|_{{}^cU} \rangle \\
        &= \langle \X^{s_0,p_0}({{}^cU}), \X^{s_1,p_1}_{\bd U}({{}^cU}) \rangle.
    \end{split}
    \end{align}
    Lemma~\ref{lem: Bullet interpolation} identifies the space on the left-hand side with either $\X^{s,p}_\bullet({{}^cU})|_{{}^cU}$ or $\W^{s,p}_\bullet({{}^cU})|_{{}^cU}$. The claim follows from Lemma~\ref{lem: Restriction HspBullet small s} in the case $s<1/p$ and from Lemma~\ref{lem: Restriction HspBullet large s} in the case $s>1/p$.
    
    The final statement on equalities in these inclusions follows from Proposition~\ref{prop: easy inclusion}.
\end{proof}


\subsection{Localization}
\label{Subsec: Localization}

We recall that $O$ satisfies a uniform Lipschitz condition around $N \coloneqq \cl{\bd O \setminus D}$ with bi-Lipschitz constant $L$ as in Definition~\ref{def: Lipschitz property}. We claim that we can select countably many points $x_i \in \cl{\bd O \setminus D}$, $i \in I \subseteq \IN \setminus \{0\}$, with corresponding coordinate charts $(U_{x_i}, \Phi_{x_i}) \eqqcolon (U_i, \Phi_i)$, and an open set $U_0$ that does not intersect $N$, with the following properties. With $J\coloneqq \{0\}\cup I$, the covering
\begin{align}
\label{eq: localization covering property}
 \cl{O} \subseteq \bigcup_{j \in J} U_j
\end{align}
admits a smooth partition of unity by functions $\eta_j \in \C^\infty(\R^d)$ satisfying
\begin{alignat*}{2}
 &\mathrm{(i)} \quad \supp(\eta_j) \subseteq U_j,  \hspace{50pt}
    &&\mathrm{(ii)} \quad \sum_{j\in J} \eta_j = 1 \text{ on } \R^d, \\
    &\mathrm{(iii)} \quad \sum_{j\in J} \mathds{1}_{U_j} \leq C \text{ on } \R^d, \hspace{50pt}
 &&\mathrm{(iv)} \quad \|\eta_j\|_{\L^\infty} + \|\nabla \eta_j \|_{\L^\infty} \leq C', \hspace{50pt}
\end{alignat*}
and there are auxiliary functions $\chi_i \in \C^\infty(\R^d)$ with $\|\chi_i\|_{\L^\infty} + \|\nabla \chi_i \|_{\L^\infty} \leq C'$ such that $\chi_i$ is $1$ on $\Phi_i(\supp \eta_i)$ and supported in $(-1,1)^{d}$, whereas $\chi_0$ is $1$ on $\supp \eta_0$ and supported in $U_0$. Here, $C$ and $C'$ are constants that depend only on $L$ and $d$.

The construction is as follows. For any $x \in N$ we extend $\Phi_x$ to a bi-Lipschitz map $\cl{U_x} \to [-1,1]^d$ with the same Lipschitz constant not larger than $L$. From $\Phi_x(x) = 0$ we conclude that $\Phi_x(\cl{U_x} \cap \B(x, \frac{1}{2L}))$ is contained in $B(0,\frac{1}{2})$ and hence does not intersect the boundary of the unit cube. The inclusion $B_x \coloneqq B(x,\frac{1}{2L}) \subseteq U_x$ then follows from the fact that bi-Lipschitz mappings between closed sets preserve the boundaries. Starting from $\bigcup_{x \in N} \frac{1}{8}B_x \supseteq N$, we use the Vitali covering lemma (Lemma~\ref{lem: Vitali}) to extract a countable collection $(x_i)_{i \in I} \subseteq N$ such that $\bigcup_{i \in I} \frac{5}{8} B_i \supseteq N$ with the $\frac{1}{8} B_i$ mutually disjoint. We have abbreviated as usual $B_i \coloneqq B_{x_i}$.

If $x \in \R^d$ is contained in $U_i$, then $U_i \subseteq\B(x_i, L\sqrt{d}) \subseteq\B(x, 2 L \sqrt{d})$ by the Lipschitz property. Due to $B(x_i, \frac{1}{16 L}) \subseteq U_i$ and mutual disjointness there are at most $(32 L^2 \sqrt{d})^d$ such $i$. Finite overlap guarantees that $U_0 \coloneqq \R^d \setminus \bigcup_{i \in I} \frac{5}{8}\cl{B}_i$ is an open set that pays for \eqref{eq: localization covering property} and we can take $C \coloneqq 1+(32 L^2 \sqrt{d})^d$ in (iii). 

For $i \in I$ we pick $\varphi_i \in \C_0^\infty(B_i)$ with range in $[0,1]$, equal to $1$ on $\frac{7}{8}B_i$, and $\|\nabla \varphi_i\|_\infty \leq c L$ for a dimensional constant $c$. We also pick a smooth $\varphi_0$ with range in $[0,1]$, support in $\R^d \setminus \bigcup_{i \in I} \frac{6}{8} B_i$, and equal to $1$ on $\R^d \setminus \bigcup_{i \in I} \frac{7}{8} B_i$. For any $x \in \R^d$ the sum $\sum_{j\in J} \varphi_j(x)$ contains at most $C$ non-zero terms, one of which is equal to $1$. Hence, functions $\eta_j$ with the properties specified in (i), (ii), (iv) are given by $\eta_j \coloneqq \varphi_j/ \sum_{j\in J} \varphi_j$. For $i \in I$ we can take the $\chi_i$ all the same since $\Phi_i(\supp(\eta_i))$ is contained in $B(0,\frac{1}{2})$. We pick $\chi_0 \in \C^\infty(U_0)$ equal to $1$ on $\R^d \setminus \bigcup_{i \in I} \frac{6}{8} B_i$ to complete the construction.

With this formalism at hand, we define the retraction-coretraction pair
\begin{align}
\label{localization E}
 \Ext: f\longmapsto \bigl(\chi_0 f, (\chi_i(f\circ \Phi_i^{-1}))_{i \in I} \bigr),
\end{align}
\begin{align}
\label{localization R}
 \Res: (g_j)_{j\in J} \longmapsto \eta_0 g_0 + \sum_{i \in I} \eta_i (g_i\circ \Phi_i).
\end{align}
Indeed, we find $\Res\Ext f=f$ for $f \in \L^p(O)$. It is implicitly understood that functions with compact support are extended by zero and domains of definitions are appropriately restricted to make these definitions meaningful.
We introduce natural function spaces for these mappings.

\begin{definition}
\label{def: IX spaces}
For $p\in (1,\infty)$ and $s\in [0,1]$ define the Banach space
\begin{align*}
    \IX^{s,p}(O) \coloneqq \X^{s,p}(O) \times \ell^p(I ; \X^{s,p}(\R^d_+)), \quad
    \|g\|_{\IX^{s,p}(O)} \coloneqq \Big(\sum_{j \in J } \|g_j\|_{\X^{s,p}}^p\Big)^{1/p}.
\end{align*}
\end{definition}

\begin{remark}
\label{rem: IX spaces interpolation}
The space $\IX^{s,p}(O)$ is constructed by $\ell^p$-superposition from $\X^{s,p}(O)$ and $\X^{s,p}(\R^d_+)$. Real and complex interpolation behaves in the best possible (componentwise) way under this operation~\cite[Sec.~1.18.1]{Triebel}. Precisely, the spaces $\IX^{s,p}(O)$ interpolate according to the same rules as do $\X^{s,p}(O)$ and $\X^{s,p}(\R^d_+)$ according to Proposition~\ref{prop: interpol X on subset complete}.
\end{remark}

\begin{lemma}
\label{lem: E and R localization}
For $p \in (1,\infty)$ and $s \in [0,1]$ the maps $\Ext: \X^{s,p}(O)\to \IX^{s,p}(O)$ and $\Res: \IX^{s,p}(O)\to \X^{s,p}(O)$ are bounded.
\end{lemma}

\begin{proof}
In view of Remark~\ref{rem: IX spaces interpolation} we only have to treat the extremal cases $s=0$ and $s=1$. For convenience we write $\IL^p(O)$ and $\IW^{1,p}(O)$ instead of $\IX^{s,p}(O)$, respectively.

Given $f \in \L^p(O)$, we use the uniformity and support properties of the partition of unity along with the uniform bi-Lipschitz property of the $\Phi_i$ when applying the transformation formula \cite[Sec.~2.3.1]{Necas2012}, to give
\begin{align}
\label{eq1: E and R localization}
 \|\Ext f \|_{\IL^p(O)}^p = \int_{O} |\chi_0 f|^p \; \d x + \sum_{i \in I} \int_{\R^d_+} |\chi_i(f\circ \Phi_i^{-1})|^p \; \d x \lesssim \int_{O} \sum_{j \in J} \mathds{1}_{U_j} |f|^p \; \d x.
\end{align}
The right-hand side is bounded by $C \|f\|_{\L^p}^p$ due to the finite overlap property (iii). Similarly, given $g \in \IL^p(O)$, we can estimate
\begin{align}
\label{eq2: E and R localization}
\begin{split}
 \|\Res g\|_{\L^p(O)}^p 
 &\leq \int_O \Big(|\eta_0 g_0| + \sum_{i \in I} |\eta_i (g_i\circ \Phi_i)| \Big)^p \; \d x
 \approx \int_O |\eta_0 g_0|^p + \sum_{i \in I} |\eta_i (g_i\circ \Phi_i)|^p \; \d x \\
 &\lesssim \int_O |g_0|^p \; \d x + \sum_{i \in I} \int_{\R^d_+} |g_i|^p \; \d x = \|g\|_{\IL^p(O)},
\end{split}
\end{align}
where in the second step we have used again that for fixed $x$ the sum contains at most $C$ non-zero terms and hence the $\ell^1$-norm can be replaced by an $\ell^p$-norm at the expense of a constant depending on $C$. The previous two estimates yield the claim in case $s=0$.

We turn to the case $s=1$ and recall that $\W^{1,p}$-spaces are defined by restriction. Let $f \in \W^{1,p}(O)$ and let $F \in \W^{1,p}(\R^d)$ be any extension. Calculating $\nabla (\Ext F)$ by the product rule and chain rules~\cite[Sec.~2.3.1]{Necas2012}, we can use the same argument as in \eqref{eq1: E and R localization} to get
\begin{align*}
 \sum_{j \in J} \|(\Ext F)_j\|_{\L^p(\R^d)}^p + \|\nabla(\Ext F)_j\|_{\L^p(\R^d)}^p \lesssim \|F\|_{\L^p(\R^d)}^p + \|\nabla F\|_{\L^p(\R^d)}^p.
\end{align*}
Since each $(\Ext F)_j$ extends $(\Ext f)_j$, the left-hand side controls $\|\Ext f\|_{\IW^{1,p}(O)}$ from above and we can pass to the infimum over $F$ to obtain the required boundedness of $\Ext$. Likewise, given $G \in \X^{s,p}(\R^d) \times \ell^p(I; \X^{s,p}(\R^d))$ we can recycle \eqref{eq2: E and R localization} to the effect that
\begin{align*}
 \|\Res G\|_{\L^p(\R^d)}^p + \|\nabla \Res G\|_{\L^p(\R^d)}^p \lesssim \sum_{j \in J} \|G_j\|_{\L^p(\R^d)}^p + \|\nabla G_j\|_{\L^p(\R^d)}^p
\end{align*}
and we conclude as before.
\end{proof}

To bring the boundary conditions into play, we introduce a modified version of $\IX^{s,p}(O)$. We set 
\begin{align}
\label{Def Ei}
 E_i\coloneqq \Phi_i(D) \cup \bigl(\R^{d-1}\setminus (-1,1)^{d-1}\bigr) \qquad (i \in I)
\end{align}
and define
\begin{align}
\label{Def IX with boundary condition}
 \IX^{s,p}_E(O) \coloneqq \X^{s,p}_{\bd O}(O) \times \bigtimes_{i \in I} \X^{s,p}_{E_i}(\R^d_+) \qquad (s > 1/p),
\end{align}
which we consider as a closed subspace of $\IX^{s,p}(O)$ in virtue of Lemma~\ref{lem: Rychkov preserves boundary conditions}. Let us make sure that these transformed Dirichlet parts are of the same geometric quality as $D$. 

\begin{lemma} \label{lem: dummes Argument 1}
    The set $E_i$ defined in \eqref{Def Ei} is $(d-1)$-regular in $\R^{d-1}$ and has porous boundary.
\end{lemma}

\begin{proof}
    Let $B\subseteq \R^{d-1}$ be a ball of radius $\r(B)\leq 1$ centered in $\cl{E_i}$. There are two cases. The first one is that $\frac{1}{2}B$ intersects the complement of $(-1,1)^{d-1}$. Then there is a ball $B'$ of radius $\r(B)/4$ contained in $B\setminus [-1,1]^{d-1}$. The second one is that $\frac{1}{2} B$ is properly contained in the domain of $\Phi_i^{-1}$ and thus there is a ball $B''\subseteq \R^d$ centered in $\bd O$ such that $\r(B)\approx \r(B'')$ and $\Phi_i^{-1}(B) \supseteq B''\cap \bd O$.

    (i) We pick the center of $B$ in $E_i$ and show that $E_i$ is $(d-1)$-regular. In the first case we have $|B\cap E_i| \gtrsim (\r(B)/4)^{d-1}$. In the second case we use that bi-Lipschitz images have comparable Hausdorff measure~\cite[Thm.~28.10 a)]{Yeh} and that $D$ is $(d-1)$-regular to conclude
    \begin{align*}
        |B\cap E_i| &\approx \cH^{d-1}(\Phi_i^{-1}(B\cap E_i)) \geq \cH^{d-1}(B'' \cap D) \gtrsim \r(B)^{d-1}.
    \end{align*}
(ii) We pick the center of $B$ in $\bd E_i$ and show that $\bd E_i$ is porous. Again, in the first case, already $B'$ does not intersect $\bd E_i$. Otherwise, we use porosity of $\bd D$ in $\bd O$, taking Remark~\ref{rem: porosity} into account, to find a ball centered in $\bd O$ and contained in $B''$ which avoids $\bd D$. Transforming this ball back using $\Phi_i$, we find a ball centered in $B$ with comparably smaller radius that does not intersect $\bd E_i$.
\end{proof}

The next lemma shows that $\Ext$ and $\Res$ defined in \eqref{localization E} and \eqref{localization R} are well-behaved with respect to the Dirichlet conditions defined in \eqref{Def Ei} and \eqref{Def IX with boundary condition}.

\begin{lemma} \label{lem: E and R localization BC}
    For $p \in (1,\infty)$ and $s \in (1/p,1]$, the operators $\Ext: \X^{s,p}_D(O)\to \IX^{s,p}_E(O)$ and $\Res: \IX^{s,p}_E(O)\to \X^{s,p}_D(O)$ are bounded.
\end{lemma}

\begin{proof}
    As for $\Ext$, it suffices to consider $f\in \C^\infty_D(O)$ since the general case follows by density, see Lemma~\ref{lem: testfunctions dense}. Since $\chi_0$ is smooth with support away from $\cl{\bd O \setminus D}$, we get that $\chi_0 f$ is smooth with compact support away from $\bd O$. In particular, we have $\chi_0\, f\in \X^{s,p}_{\bd O}(O)$. We conclude from the bi-Lipschitz property of $\Phi_i$ that 
    \begin{align*}
     \dist(E_i, \supp(f\circ \Phi_i^{-1}))=\dist(\Phi_i(D),\Phi_i(\supp f)) \approx \dist(D, \supp f) >0.
    \end{align*}
    Hence, $\chi_i(f\circ \Phi_i^{-1})$ is a Lipschitz continuous function on $\R^d_+$ whose compact support has positive distance to $E_i$. Thus, it is contained in $\W^{1,p}_{E_i}(\R^d_+) \subseteq \X^{s,p}_{E_i}(\R^d_+)$.
 
    As for $\Res$, we take $g=(g_j)_{j \in J}$ from $\C^\infty_{\bd O}(O) \times \bigtimes_{i \in I} \C^\infty_{E_i}(\R^d_+)$, which is dense in $\IX^{s,p}_E(O)$ due to Lemmas~\ref{lem: testfunctions dense} and \ref{lem: dummes Argument 1}. As before, we only have to show that the support of $\Res g$ has positive distance to $D$. But $\supp(\eta_0\, g_0) \subseteq \supp(g_0)$ has positive distance to $D$ by construction and for $\supp(\eta_i(g_i \circ \Phi_i))$ we can argue as above.
\end{proof}

We formulate a reduction result based on this localization.

\begin{proposition} 
\label{prop: reduction by localization}
    The set inclusions \eqref{Ip complex} and \eqref{Ip real} follow from the set inclusions
    \begin{align}
        \label{half-space Ip complex}
        [\X^{s_0,p_0}(\R^d_+), \X^{s_1,p_1}_{E_i}(\R^d_+)]_\theta \supseteq \begin{cases}
                                           \X^{s,p}_{E_i}(\R^d_+) &(\text{if $s > 1/p$}) \\
                                           \X^{s,p}(\R^d_+) &(\text{if $s < 1/p$})
                                          \end{cases}
    \end{align}
    and
    \begin{align}
        \label{half-space Ip real}
        (\X^{s_0,p_0}(\R^d_+), \X^{s_1,p_1}_{E_i}(\R^d_+))_{\theta,p} \supseteq \begin{cases}
                                           \W^{s,p}_{E_i}(\R^d_+) &(\text{if $s > 1/p$}) \\
                                           \W^{s,p}(\R^d_+) &(\text{if $s < 1/p$})
                                          \end{cases}.
    \end{align}
\end{proposition}

\begin{proof}
    We apply Proposition~\ref{prop: retraction-coretraction} with the pair $(\Ext, \Res)$ defined in \eqref{localization E} and \eqref{localization R}. Owing to the mapping properties derived in Lemmas~\ref{lem: E and R localization} and  \ref{lem: E and R localization BC}, we get equal sets
    \begin{align*}
        [\X^{s_0,p_0}(O),\X^{s_1,p_1}_D(O)]_\theta = \Res [\IX^{s_0,p_0}(O),\IX^{s_1,p_1}_E(O)]_\theta.
    \end{align*}
    Lemma~\ref{lem: E and R localization BC} asserts that the inclusion \eqref{Ip complex} holds provided that we can prove
    \begin{align}
    \label{eq1: reduction by localization}
        [\IX^{s_0,p_0}(O), \IX^{s_1,p_1}_E(O)]_\theta \supseteq \begin{cases}
                                           \IX^{s,p}_E(O) &(\text{if $s > 1/p$}) \\
                                           \IX^{s,p}(O) &(\text{if $s < 1/p$})
                                          \end{cases}.
    \end{align}
    The $\ell^p$-superpositions of spaces $\X^{s,p}(O)$ and $\X^{s, p}_{E_i}(\R^d_+)$ on the left interpolate componentwise, see Remark~\ref{rem: IX spaces interpolation}.
    This being said, the above follows from the assumption \eqref{half-space Ip complex} for the components on $\R^d_+$ and Proposition~\ref{prop: pure Dirichlet interpolation} for the component on $O$. 
    
    The real case is the same upon using $\IW$-spaces on the right of \eqref{eq1: reduction by localization} and appealing to assumption \eqref{half-space Ip real} instead.
\end{proof}

\begin{remark}
\label{rem: reduction by localization}
It stems from the interpolation on the left-hand side of \eqref{eq1: reduction by localization} that at least at this stage of the proof we prefer talking about set inclusions only. Continuity of \eqref{eq1: reduction by localization} would require continuity of \eqref{half-space Ip complex} (which we could obtain) -- but with uniform bounds in $I$ (which we believe to be rather painful).
\end{remark}

\subsection{Extension and restriction operators for the half-space} \label{Subsec: Ext-Res halfspace}

We introduce the extension and restriction operators appearing in Figure~\ref{fig: road map complex}. As usual, $\X$ denotes either $\H$ or~$\W$. 

\subsubsection*{The restriction operator $\Res$}

Let $F\subseteq \bd \R^d_+$ be $(d-1)$-regular. We identify $\bd \R^d_+$ with $\R^{d-1}$ whenever convenient. Proposition~\ref{prop: JW} yields a restriction operator $\Res_F: \X^{s,p}(\R^d)\to \W^{s-1/p,p}(F)$ for $p \in (1,\infty)$ and $s\in (1/p,1+1/p)$.  By construction, we have for $u\in \X^{s,p}(\R^d) \cap \C(\R^d)$,
\begin{align} \label{Eqn: R on continuous functions}
    \Res_F u(x') = u(x',0) \qquad (\text{a.e.\ $x' \in F$}).
\end{align}
In virtue of this formula $\Res_F$ is well-defined on the quotient space $\X^{s,p}(\R^d_+)\,\cap \,\C(\cl{\R^d_+})$. The inclusion chain
\begin{align}
    \C^\infty_0(\R^d)|_{\R^d_+} \subseteq \X^{s,p}(\R^d_+) \cap \C(\cl{\R^d_+}) \subseteq \X^{s,p}(\R^d_+) = \X^{s,p}(\R^d)|_{\R^d_+}
\end{align}
and the density of the first space in the last space shows that we can extend $\Res_F$ to $\X^{s,p}(\R^d_+)$ by continuity. We abbreviate $\Res\coloneqq \Res_{\bd \R^d_+}$.

\subsubsection*{The extension operator $\Ext$} 

For the extension operator we also need to consider spaces of negative smoothness. They have been defined on the whole space in Section~\ref{Subsec: Not Spaces}. We set $\X^{s,p}(\R^d_+) \coloneqq \X^{s,p}(\R^d)|_{\R^d_+}$, where the restriction of distributions $|_{\R^d_+}: \S'(\R^d) \to \cD'(\R^d_+)$ coincides with the pointwise restriction when $s$ is non-negative.

We construct $\Ext$ via the \emph{bounded analytic $C_0$-semigroup} $(\e^{-\Lambda t})_{t \geq 0}$ generated by $\Lambda \coloneqq -(1-\Delta_{x'})^{1/2}$ in $\L^p(\R^{d-1})$. Here, $\Delta_{x'}$ denotes the \emph{Laplacian} in $\R^{d-1}$. A reader who is not familiar with these notions may consult the textbook \cite{ABHN}, in particular Example~3.7.6 and Theorem~3.8.3. By means of the Fourier transform $\cF$ in $\R^{d-1}$ the operators $\e^{-\Lambda t}$ are unambiguously defined on all of $\S'(\R^{d-1})$ through
\begin{align*}
 \e^{-\Lambda t}: \S'(\R^{d-1}) \to \cD'(\R^d_+), \quad u \mapsto \cF^{-1} \big(\e^{-t \sqrt{1+|\xi'|^2}} \cF u(\xi')\big).
\end{align*}
We write $\dom_p(\Lambda^k)$ for the maximal domain of $\Lambda^k$ in $\L^p(\R^d)$ and equip it with the graph norm $\|\cdot\|_{\L^p} + \|\Lambda^k \cdot \|_{\L^p}$. By definition of Bessel potential spaces, we have for $k \in \IN$ up to equivalent norms,
\begin{align}
\label{Eqn: domain poisson semigroup generator}
    \dom_{p}(\Lambda^k)= \H^{k,p}(\R^{d-1}) = \W^{k,p}(\R^{d-1}).
\end{align}
Abstract semigroup theory~\cite[Thm.~1.14.5]{Triebel} provides an equivalent norm on the real interpolation space $\W^{k-1/p,p} = (\L^p, \H^{k,p})_{1-\frac{1}{kp},p}$:
\begin{align}    
\label{Eqn: poisson semigroup derivative estimate}
    \|\partial_t^k \e^{-\Lambda t} u\|_{\L^p(\R_+;\L^p(\R^{d-1}))} \approx \|u\|_{\W^{k-1/p,p}(\R^{d-1})} \qquad (u \in \W^{k-1/p,p}(\R^{d-1})).
\end{align}
With this at hand, we fix $\chi\in \C_0^\infty(\cl{\R_+})$ with $\chi(0)=1$ and define the operator
\begin{align}
\label{eq: def extension half space}
    \Ext u(x',x_d) \coloneqq \chi(x_d) \e^{-\Lambda x_d} u(x') \qquad (x'\in \R^{d-1}, x_d \geq 0).
\end{align}

\begin{proposition} \label{prop: half-space extension operator}
    The operator $\Ext$ defined in \eqref{eq: def extension half space} is $\W^{s,p}(\R^{d-1}) \to \X^{s+1/p,p}(\R^d_+)$-bounded for all $p\in (1,\infty)$ and $s\in \R \setminus \IZ$.
\end{proposition}

\begin{proof}
 Our argument is an adaption of \cite[Sec.~2.9.3]{Triebel} and divides into six steps.
 
\emph{Step 1: $\Ext: \W^{1-1/p,p}(\R^{d-1})\to \W^{1,p}(\R^d_+)$ is bounded}. First, we note that for $\varphi\in \C_0^\infty(\cl{\R_+})$ the multiplication operator
\begin{align}
    \varphi(x_d): \L^p(\R_+;\L^p(\R^{d-1})) \to \L^p(\R_+;\L^p(\R^{d-1})) \label{Eqn: cutoff multiplier}
\end{align}
is bounded. By boundedness of the semigroup the same is true for
\begin{align}
    \varphi(x_d) \e^{-\Lambda x_d}: \L^p(\R^{d-1}) \to \L^p(\R_+;\L^p(\R^{d-1})). \label{Eqn: cutoff semigroup}
\end{align}
In particular, we get $\Ext: \L^p(\R^{d-1}) \to \L^p(\R^d_+)$ if we choose $\varphi=\chi$ in \eqref{Eqn: cutoff semigroup}. Using the product rule and \eqref{Eqn: cutoff multiplier}, we deduce from \eqref{Eqn: poisson semigroup derivative estimate} that for $k\in \IN$ we have
\begin{align}
    \|\partial_d^k \Ext u\|_{\L^p(\R^d_+)} \lesssim \|u\|_{\W^{k-1/p}(\R^{d-1})}. \label{Eqn: x_d derivative estimates}
\end{align}
Using \eqref{Eqn: domain poisson semigroup generator}, \eqref{Eqn: cutoff multiplier}, the identity $\Lambda \e^{-\Lambda x_d} = -\partial_d \e^{-\Lambda x_d}$ and \eqref{Eqn: x_d derivative estimates}, we obtain
\begin{align*}
    \|\Ext u\|_{\W^{1,p}(\R^d_+)} &\approx \|\Ext u\|_{\L^p(\R_+;\W^{1,p}(\R^{d-1}))} + \|\Ext u\|_{\W^{1,p}(\R_+;\L^p(\R^{d-1}))} \\
                                  &\approx \|\Ext u\|_{\L^p(\R^d_+)} + \|\chi(x_d) \Lambda \e^{-\Lambda x_d} u(x')\|_{\L^p(\R^d_+)} + \|\partial_d \Ext u\|_{\L^p(\R^d_+)} \\
                                  &\lesssim \|u\|_{\W^{1-1/p,p}(\R^{d-1})}.
\end{align*}

\emph{Step 2: $\Ext: \W^{k-1/p,p}(\R^{d-1}) \to \W^{k,p}(\R^d_+)$ is bounded for $k\in \IN$}. We argue by induction. The case $k=1$ was treated in Step~1. Moreover, the derivatives in $x_d$-direction are under control owing to \eqref{Eqn: x_d derivative estimates}. We fix $1\leq j\leq d-1$. As $\e^{-\Lambda x_d}$ and $\partial_j$ both are Fourier multipliers on $\S'(\R^{d-1})$, they commute. Assume the claimed boundedness holds for $k\in \IN$. Then $\partial_j^k \Ext: \W^{k-1/p,p}(\R^{d-1}) \to \L^p(\R^d_+)$ is bounded and we conclude from Lemma~\ref{lem: Xsp norm with gradient} that
\begin{align}
    \|\partial_j^{k+1} \Ext u\|_{\L^p(\R^d_+)} = \|\partial_j^k \Ext \partial_j u\|_{\L^p(\R^d_+)} \lesssim \|\partial_j u\|_{\W^{k-1/p,p}(\R^{d-1})} \lesssim \|u\|_{\W^{k+1-1/p,p}(\R^{d-1})}.
\end{align}

\emph{Step 3: Lifting property}. To bring negative orders of differentiability into play, we introduce for $m \in \IN$ the lift operator $\cI_{2m}\coloneqq \cF^{-1} (1+|\xi'|^2)^{-m} \cF$ defined on $\S'(\R^{d-1})$. It is invertible and we write $\cI_{-2m}\coloneqq \cI_{2m}^{-1}$. For $s\in \R$ the operator $\cI_{2m}$ is an isomorphism $\H^{s,p}(\R^{d-1})\to \H^{s+2m,p}(\R^{d-1})$ by definition of the norms on Bessel potential spaces. Since the Fourier multipliers $\Ext$ and $\cI_{-2m}$ commute, we can decompose
\begin{align}
    \Ext = \cI_{-2m} \circ \Ext \circ \cI_{2m}, \label{Eqn: Ext decomposition}
\end{align}
in order to lift the argument of $\Ext$ into a space with positive order of differentiability.

\emph{Step 4: $\cI_{-2m}$ in $d$-dimensional space}. Since $\cI_{-2m} = (1-\Delta_{x'})^m$ is a differential operator of order $2m$ acting only in $d-1$ coordinates, we have $\cI_{-2m}: \H^{s+2m,p}(\R^d)\to \H^{s,p}(\R^d)$ for integer $s$. Interpolation by means of Proposition~\ref{prop: interpolation whole space} yields $\cI_{-2m}: \X^{s+2m,p}(\R^d)\to \X^{s,p}(\R^d)$. The differential operator $\cI_{-2m}$ is local in the sense that it commutes with the distributional restriction. Hence, its restriction to the upper half-space is well-defined and we get
\begin{align}
    \cI_{-2m}: \X^{s+2m,p}(\R^d_+)\to \X^{s,p}(\R^d_+). \label{Eqn: negative lift on Bessel}
\end{align}

\emph{Step 5: Interpolation of $\cI_{2m}$ and $\Ext$}. As before, we interpolate $\cI_{2m}: \H^{s,p}(\R^{d-1})\to \H^{s+2m,p}(\R^{d-1})$ from Step~3 to obtain for all $s \in \R$ boundedness of 
\begin{align}
    \cI_{2m}: \W^{s,p}(\R^{d-1})\to \W^{s+2m,p}(\R^{d-1}). \label{Eqn: positive lift on Besov}
\end{align}
Similarly, real and complex interpolation of the outcome of Step~2 with the aid of Proposition~\ref{prop: interpol X on subset} yields 
\begin{align}
    \Ext: \W^{s,p}(\R^{d-1}) \to \X^{s+1/p,p}(\R^d_+) \label{Eqn: Ext interpolation}
\end{align}
if $s \geq 1-1/p$ is not an integer.

\emph{Step 6: Patching everything together}. Let $s\in \R \setminus \IZ$. If $s\geq 1-1/p$, then $\Ext: \W^{s,p}(\R^{d-1}) \to \X^{s+1/p,p}(\R^d_+)$ follows by \eqref{Eqn: Ext interpolation}. Otherwise, we choose $m\in \IN$ such that $2m+s \geq 1-1/p$. We use the decomposition \eqref{Eqn: Ext decomposition} to conclude $\Ext: \W^{s,p}(\R^{d-1}) \to \X^{s+1/p,p}(\R^d_+)$ from \eqref{Eqn: positive lift on Besov}, \eqref{Eqn: Ext interpolation} and \eqref{Eqn: negative lift on Bessel}.
\end{proof}

The next lemma justifies calling $\Ext$ an extension operator.

\begin{lemma} \label{Lem: restriction of extension on F}
    Let $F\subseteq \bd\R^d_+$ be $(d-1)$-regular and $\Res_F$ the corresponding restriction operator. Let $p\in (1,\infty)$ and suppose that $s>0$ is not an integer. If $u\in \W^{s,p}(\R^{d-1})$, then $\Res_F \Ext u = u$ holds almost everywhere on $F$ 
\end{lemma}

\begin{proof}
    By density it suffices to prove the claim for $u\in \C^\infty_0(\R^{d-1})$. Due to \eqref{Eqn: domain poisson semigroup generator} we have $u \in \dom_p(\Lambda^k)$ for all $k \in \IN$ and $p \in (1,\infty)$. We pick $k$ and $p$ such that $\dom_p(\Lambda^k)$ is continuously included into $\C(\R^{d-1})$ in virtue of Sobolev embeddings. Since we have $\Lambda^k \e^{-t \Lambda}u = \e^{-t\Lambda} \Lambda^k u$ for $t \geq 0$, the strong continuity of the semigroup on $\L^p(\R^{d-1})$ implies  $\Ext u \in \C(\cl{\R^+}; \C(\R^{d-1})) = \C(\cl{\R^d_+})$ and $\Ext u(x',0) = u(x')$ for almost every $x' \in \R^{d-1}$. Proposition~\ref{prop: half-space extension operator} guarantees $\Ext u \in \X^{s+1/p,p}(\R^d_+)$ and we conclude from \eqref{Eqn: R on continuous functions} that $\Res_F \Ext u = u$ holds almost everywhere on $F$. 
\end{proof}

\subsection{Conclusion of the proof}
\label{Subsec: conclusion}

Here, we will verify the set inclusions \eqref{half-space Ip complex} and \eqref{half-space Ip real}. Thereby we complete the proof of Theorem~\ref{thm: main result}.

We start out with the interpolation in the case $s\in (0,1/p)$, which we treat slightly more generally for a later use.

\begin{proposition} \label{prop: mixed interpolation small s}
    Let $p_0,p_1 \in (1,\infty)$, $s_0 \in [0,1/p_0)$, $s_1 \in (1/p_1,1]$, and for $\theta \in (0,1)$ define $p$ and $s$ as in \eqref{eq: interpolating paramaters}.
    Suppose $s<1/p$. Assume that $U \subseteq \R^d$ is a closed $d$-regular set with $(d-1)$-regular boundary. Moreover, assume that ${}^c U$ is also $d$-regular and that $F\subseteq \bd U$ is $(d-1)$-regular. Then it follows up to equivalent norms that
    \begin{align}
       \tag{i}\label{MixedSmall1} [\X^{s_0,p_0}({}^c U), \X_{F}^{s_1,p_1}({}^c U)]_{\theta} &= \X^{s,p}({}^c U),
                                                                    \\[8pt]
       \tag{ii}\label{MixedSmall2}\noeqref{MixedSmall2} (\X^{s_0,p_0}({}^c U), \X_{F}^{s_1,p_1}({}^cU))_{\theta,p} &= \W^{s,p}({}^c U),
    \end{align}
    with the exception that $s_0 \neq 0$ and $s_1 \neq 1$ are required in \eqref{MixedSmall1} for $\X = \W$.
\end{proposition}

\begin{proof}
    The ``$\subseteq$''-inclusions follow from Proposition~\ref{prop: easy inclusion}. For the converse let $\bracket$ denote either the $\theta$-complex or $(\theta,p)$-real interpolation bracket. Using the inclusion $\X^{s_1,p_1}_{\bd U}({}^cU) \subseteq \X^{s_1,p_1}_F({}^c U)$, we get
    \begin{align}
	\langle \X^{s_0,p_0}({}^c U), \X^{s_1,p_1}_{\bd U}({}^c U) \rangle \subseteq \langle \X^{s_0,p_0}({}^cU), \X^{s_1,p_1}_F({}^cU) \rangle.
    \end{align}
    We identify the space on the left-hand side according to Proposition~\ref{prop: pure Dirichlet interpolation} to conclude.
\end{proof}

Since this proposition can be applied to $U\coloneqq \cl{\R^d_-}$ and $F\coloneqq E_i$, we get \eqref{half-space Ip complex} and \eqref{half-space Ip real} in case $s<1/p$.

In a next step we establish the rest of Figure~\ref{fig: road map complex}. To this end, we shall appeal to the theory of Section~\ref{Subsec: Bullet spaces} with $U = E_i$ in $\R^{d-1}$. This requires $E_i$ to be $(d-1)$-regular in $\R^{d-1}$, its boundary to be a Lebesgue null set, and its interior to be of some class $\cD^t$. The first requirement is met by Lemma~\ref{lem: dummes Argument 1}, which also guarantees that $\bd E_i$ is porous. Hence, so is its subset $\bd E_i^\circ$. In view of Example~\ref{ex: examples Dt class} the interior of $E_i$ is of class $\cD^t$ for some $t \in (0,1)$. Finally, the boundary of a porous set is a null set by Lemma~\ref{lem: porous boundary null set}.

Due to Lemma~\ref{lem: Bullet interpolation} the spaces $\W^{s,p}_\bullet({}^cE_i)$ interpolate as expected. Next, we check that the extension operator constructed in the previous section preserves the zero condition when restricted to $\W^{s,p}_\bullet({}^cE_i)$.

\begin{lemma} \label{lem: Ext preserves zero condition}
    Let $p \in (1,\infty)$ and let $s> t(1/p -1)$ not be an integer. If $\X$ denotes either $\H$ or $\W$, then $\Ext: \W^{s,p}_\bullet({}^c{E_i}) \to \X^{s+1/p,p}_{E_i}(\R^d_+)$.
\end{lemma}

\begin{proof}
    Let $u\in \W^{s,p}_\bullet({}^cE_i)$. Due to Proposition~\ref{prop: half-space extension operator} we have $\Ext u \in \X^{s+1/p,p}(\R^d_+)$. By Lemma~\ref{Lem: restriction of extension on F} and the definition of $\W^{s,p}_\bullet({}^cE_i)$ we know that $\Res_{E_i} \Ext u = u = 0$ holds. This means $\Ext u \in \X^{s+1/p,p}_{E_i}(\R^d_+)$.
\end{proof}

Let now $p_0, p_1 \in (1,\infty)$, $s_0 \in [0,1/p_0)$, $s_1 \in (1/p_1,1]$, and $\theta \in (0,1)$. Let us recall
\begin{align*}
 \frac{1}{p} = \frac{1-\theta}{p_0} + \frac{\theta}{p_1}, \quad s = (1-\theta)s_0 + \theta s_1
\end{align*}
and that we assume $s>1/p$. By these restrictions on the parameters there are $q\in (1,\infty)$ and $\eps \in (0, \min\{1/q, t(1-1/q)\})$ such that the point $(1/q,-\varepsilon)^\top$ lies on the segment connecting $(1/p_1,s_1-1/p_1)^\top$ and $(1/p_0,s_0-1/p_0)^\top$ in the $(1/p,s)$-plane. Since we have by construction
\begin{align*}
  \begin{pmatrix} 1/p \\ s-1/p \end{pmatrix} = (1-\theta) \begin{pmatrix} 1/p_0 \\ s_0 - 1/p_0 \end{pmatrix} + \theta \begin{pmatrix} 1/p_1 \\ s_1-1/p_1 \end{pmatrix},
\end{align*}
we can fix $\eta\in (0,\theta)$ such that
\begin{align}
    \begin{pmatrix} 1/p \\ s-1/p \end{pmatrix} = (1-\eta) \begin{pmatrix} 1/q \\ -\varepsilon \end{pmatrix} + \eta \begin{pmatrix} 1/p_1 \\ s_1-1/p_1 \end{pmatrix}.
\end{align}
This yields identity $(\heartsuit)$ in Figure~\ref{fig: road map complex}. Adding both lines of the previous equation gives
\begin{align}
    s = (1-\eta)(1/q-\varepsilon)+\eta s_1.
\end{align}
We deduce
\begin{align}
    \Bigl(1-\frac{\theta-\eta}{1-\eta}\Bigr) s_0 + \frac{\theta-\eta}{1-\eta} s_1 = 1/q-\varepsilon.
\end{align}
In the following all function spaces are on $\R^d_+$ and we omit the dependence. Let $\langle \cdot\,, \cdot \rangle$ denote either the complex or the $(\cdot\,,p)$-real interpolation bracket. From Proposition~\ref{prop: mixed interpolation small s} and Proposition~\ref{prop: 1-sided reiteration} we deduce
\begin{align}
    \langle \X^{1/q-\varepsilon,q}, \X^{s_1,p_1}_{E_i} \rangle_\eta 
    = \langle[\X^{s_0,p_0}, \X^{s_1,p_1}_{E_i}]_\frac{\theta-\eta}{1-\eta}, \X^{s_1,p_1}_{E_i}\rangle_\eta
    = \langle \X^{s_0,p_0}, \X^{s_1,p_1}_{E_i} \rangle_\theta,
\end{align}
where $s_0 \neq 0$ and $s_1 \neq 1$ are required in case $\X = \W$. This establishes Figure~\ref{fig: road map complex} in case of complex interpolation. It also establishes the analogue that corresponds to real interpolation of $\H$-spaces. As for real interpolation of $\W$-spaces, we invoke the following reiteration theorem~\cite[Thm.~3.5.3]{BL}. 

\begin{proposition}
\label{prop: reiteration real}
Let $(X_0,X_1)$ be an interpolation couple. Let $p \in [1,\infty]$, $\theta_0, \theta_1 \in [0,1]$ with $\theta_0 \neq \theta_1$, and $\lambda \in (0,1)$. With $\theta \coloneqq (1-\lambda)\theta_0 + \lambda \theta_1$ it follows that up to equivalent norms
\begin{align*}
 ((X_0,X_1)_{\theta_0,p}, (X_0,X_1)_{\theta_1,p})_{\lambda,p} = (X_0, X_1)_{\theta,p},
\end{align*}
subject to the interpretation $(X_0,X_1)_{j,p} \coloneqq X_j$ in the endpoint cases $j \in \{0,1\}$.
\end{proposition}

Indeed, in combination with Proposition~\ref{prop: mixed interpolation small s} we can give
\begin{align}
     (\W^{1/q-\varepsilon,q}, \W^{s_1,p_1}_{E_i})_{\eta,p} 
    = ((\W^{s_0,p_0}, \W^{s_1,p_1}_{E_i})_{\frac{\theta-\eta}{1-\eta},p}, \W^{s_1,p_1}_{E_i})_{\eta,p} 
    = (\W^{s_0,p_0}, \W^{s_1,p_1}_{E_i})_{\theta,p}
\end{align}
without requiring $s_0 \neq 0$ or $s_1 \neq 1$. This completes Figure~\ref{fig: road map complex} in the remaining case.

With this at hand, we complete the proof of Theorem~\ref{thm: main result}. Let $\langle \cdot\,, \cdot \rangle$ denote either the complex or the $(\cdot\,,p)$-real interpolation bracket. With Lemma~\ref{Lem: restriction of extension on F} we derive $\Res(f-\Ext\Res f)=0$ for $f\in \X^{s,p}$, which means $f-\Ext\Res f\in \X^{s,p}_{\bd \R^d_+}$. We have
\begin{align}
     \langle\X^{s_0,p_0},\X^{s_1,p_1}_{\bd \R^d_+}\rangle_\theta \subseteq \langle\X^{s_0,p_0},\X^{s_1,p_1}_{E_i}\rangle_\theta,
\end{align}
where Proposition~\ref{prop: pure Dirichlet interpolation} identifies the left-hand space as $\X^{s,p}_{\bd \R^d_+}$ for complex interpolation and as $\W^{s,p}_{\bd \R^d_+}$ for real interpolation. From the decomposition
\begin{align*}
 f=(f-\Ext\Res f)+\Ext\Res f
\end{align*}
we conclude $f\in [\X^{s_0,p_0},\X^{s_1,p_1}_{E_i}]_\theta$ for $f\in \X^{s,p}_{E_i}$ in case of complex interpolation, which completes the proof of \eqref{Ip complex}, and $f\in (\X^{s_0,p_0}, \X^{s_1,p_1}_{E_i})_{\theta,p}$ for $f\in \W^{s,p}_{E_i}$, which shows \eqref{Ip real}.

\section{A complex \texorpdfstring{($\W^{-1,p}_D$,$\W^{1,p}_D$)}{W1p-W-1p} interpolation formula}
\label{Sec: W-1p}

In this section we prove Theorem~\ref{thm: W1p-W-1p}. We begin by defining spaces of negative smoothness with boundary conditions on an open set.

\begin{definition}
\label{def: X negative smoothness}
Let $O \subseteq \R^d$ be open and $D \subseteq \cl{O}$ be $(d-1)$-regular. Let $p \in (1,\infty)$ and $s \in [0,1]$. For $\X$ either $\H$ or $\W$ define
\begin{align*}
 \X^{-s,p}(O) \coloneqq (\X^{s,p'}(O))^*
\end{align*}
and if $s > 1-1/p$ define
\begin{align*}
 \X_D^{-s,p}(O) \coloneqq (\X_D^{s,p'}(O))^*.
\end{align*}
\end{definition}

The second part of the definition is consistent with the case $O = \R^d$, see Section~\ref{Subsec: Not Spaces}.

We are concerned with interpolation spaces between $\W^{-1,p}_D(O)$ and $\W^{1,p}_D(O)$. These two spaces form an interpolation couple since we can naturally view $\W^{1,p}_D(O)$ as a subspace of $\W^{-1,p}_D(O)$ by extending the $\L^p(O) - \L^{p'}(O)$ duality. We also recall that as a consequence of Lemma~\ref{lem: testfunctions dense} the inclusion $\W^{1,p}_D(O) \subseteq \L^p(O)$ is dense. 

As for interpolation of dual spaces, we have the following principle~\cite[Cor.~4.5.2]{BL}, see also \cite[Cor.~2.15]{Bechtel} for a proper treatment of the spaces of conjugate-linear functionals indicated by a superscript asterisk.

\begin{proposition}
\label{prop: duality complex interpolation}
Let $(X_0,X_1)$ be an interpolation couple such that $X_0 \cap X_1$ is dense in both $X_0$ and $X_1$ and assume that $X_0$ is reflexive. For $\theta \in (0,1)$ it follows that with equal norms
\begin{align*}
 [X_1^*, X_0^*]_{1-\theta} = ([X_0, X_1]_\theta)^*.
\end{align*}
\end{proposition}

The idea of proof is to patch together the interpolation scale provided by Theorem~\ref{thm: main result} with its dual scale. This requires some overlap of interpolation scales. The following lemmas use some notions introduced in Section~\ref{Subsec: Symmetric interpolation}.

\begin{lemma}
\label{lem: negative s on O}
Let $O \subseteq \R^d$ be an open set with $(d-1)$-regular boundary. Let $p \in (1,\infty)$, $s \in (1/p-1,1/p)$, and let $\X$ denote either $\H$ or $\W$. There is a retraction $\Res: \X^{s,p}(\R^d) \to \X^{s,p}(O)$ with corresponding coretraction $\Ext: \X^{s,p}(O) \to \X^{s,p}(\R^d)$. These operators are the same for all $p$ and $s$.
\end{lemma}

\begin{proof}
If $s \in [0,1/p)$, then due to Corollary~\ref{cor: zero extension bounded} we can take $\Res \coloneqq |_O$ and $\Ext \coloneqq \Ext_0$ the extension by $0$. By the usual identification of functions with distributions, these operators consistently act on $\X^{s,p}$ also when $s \in (1/p-1,0]$. Indeed, if $f \in \X^{s,p}(\R^d)$ then for all $\varphi \in \X^{-s,p'}(O)$ we set
\begin{align*}
 \langle f|_O, \varphi \rangle_{\X^{s,p}(O), \X^{-s,p'}(O)} \coloneqq \langle f, \Ext_0 \varphi \rangle_{\X^{s,p}(\R^d), \X^{-s,p'}(\R^d)},
\end{align*}
where $\langle \cdot\, , \cdot \rangle$ denotes the respective duality pairing. Well-definedness and boundedness of $|_O: \X^{s,p}(\R^d) \to \X^{s,p}(O)$ follows again from Corollary~\ref{cor: zero extension bounded}. Conversely, given $g \in \X^{s,p}(O)$, we let the zero extension $\Ext_0 g$ act on $\psi \in \X^{-s,p'}(\R^d)$ via
\begin{align*}
 \langle \Ext_0 g, \psi \rangle_{\X^{s,p}(\R^d), \X^{-s,p'}(\R^d)} \coloneqq \langle g, \psi|_O \rangle_{\X^{s,p}(O), \X^{-s,p'}(O)}.
\end{align*}
It is bounded since $|_O: \X^{-s,p'}(\R^d) \to \X^{-s,p'}(O)$ is bounded by definition of the quotient norm. Finally, $(\Ext_0 g)|_O = g$ follows by concatenating the two identities above.
\end{proof}

\begin{lemma}
\label{lem: interpolation on O around s=0}
Let $O \subseteq \R^d$ be an open set with $(d-1)$-regular boundary. Let $p \in (1,\infty)$, $s_0,s_1 \in (1/p-1,1/p)$, $\theta \in (0,1)$, and set $s \coloneqq (1-\theta)s_0 + \theta s_1$. If $\X$ denotes either $\H$ or $\W$, then up to equivalent norms
\begin{align*}
 [\X^{s_0, p}(O), \X^{s_1,p}(O)]_{\theta} = \X^{s,p}(O)
\end{align*}
with the exception that $s=0$ is only allowed if $\X=\H$.
\end{lemma}

\begin{proof}
The corresponding identities on $O = \R^d$ are due to Proposition~\ref{prop: interpolation whole space}. The conclusion follows from Proposition~\ref{prop: retraction-coretraction} applied with the retraction-coretraction pair from Lemma~\ref{lem: negative s on O}.
\end{proof}

With these tools at hand, we can give the

\begin{proof}[Proof of Theorem~\ref{thm: W1p-W-1p}]
We appeal to Wolff's result, Proposition~\ref{prop: Wolff}. All function spaces will be on $O$ and we omit the dependence. We fix some $s \in (0, \min\{1/p, 1-1/p\})$ and consider the following diagram.
\begin{center}
\begin{tikzpicture}[scale=0.7]
\coordinate[label=left:$\H^{-s,p}$] (X0) at (0,0);
\coordinate[label=left:$\L^{p}$] (Xtheta) at (3,0);
\coordinate[label=right:$\H^{s,p}$] (Xeta) at (5,0);
\coordinate[label=right:$\W^{1,p}_D$] (X1) at (9,0);
\coordinate[label=above:${[\cdot\,, \cdot]_{1/2}}$] (I1) at (3,0.7);
\coordinate[label=below:${[\cdot\,, \cdot]_{s}}$] (I1) at (5,-0.7);

\draw (0,0) -- (0,0.7);
\draw (0,0.7) -- (5,0.7);
\draw (5,0.7) -- (5,0);
\draw[dashed] (3,0) -- (3,0.7);

\draw (3,0) -- (3,-0.7);
\draw (3,-0.7) -- (9,-0.7);
\draw (9,-0.7) -- (9,0);
\draw[dashed] (5,-0.7) -- (5,0);

\fill (X0) circle (3pt);
\fill (Xtheta) circle (3pt);
\fill (Xeta) circle (3pt);
\fill (X1) circle (3pt);
\end{tikzpicture}
\end{center}
The $1/2$-interpolation is due to Lemma~\ref{lem: interpolation on O around s=0} and $s$-interpolation is due to Theorem~\ref{thm: main result}. Proposition~\ref{prop: Wolff} yields $\L^p = [\H^{-s,p}, \W^{1,p}_D]_{s/(1+s)}$. Therefore we can consider the diagram
\begin{center}
\begin{tikzpicture}[scale=0.7]
\coordinate[label=left:$\W^{-1,p}_D$] (X0) at (0,0);
\coordinate[label=left:$\H^{-s,p}$] (Xtheta) at (3,0);
\coordinate[label=right:$\L^p$] (Xeta) at (5,0);
\coordinate[label=right:$\W^{1,p}_D$] (X1) at (9,0);
\coordinate[label=above:${[\cdot\,, \cdot]_{1-s}}$] (I1) at (3,0.7);
\coordinate[label=below:${[\cdot\,, \cdot]_{s/(1+s)}}$] (I1) at (5,-0.7);

\draw (0,0) -- (0,0.7);
\draw (0,0.7) -- (5,0.7);
\draw (5,0.7) -- (5,0);
\draw[dashed] (3,0) -- (3,0.7);

\draw (3,0) -- (3,-0.7);
\draw (3,-0.7) -- (9,-0.7);
\draw (9,-0.7) -- (9,0);
\draw[dashed] (5,-0.7) -- (5,0);

\fill (X0) circle (3pt);
\fill (Xtheta) circle (3pt);
\fill (Xeta) circle (3pt);
\fill (X1) circle (3pt);
\end{tikzpicture}
\end{center}
where the $(1-s)$-interpolation follows from Theorem~\ref{thm: main result} by means of the duality principle of Proposition~\ref{prop: duality complex interpolation}. Another application of Proposition~\ref{prop: Wolff} completes the proof.
\end{proof}
\section{Real interpolation via the trace method}
\label{Sec: Real}

Here, we present the proof of Theorem~\ref{thm: real interpolation via trace}.

\subsection{Road map}
\label{Subsec: Roadmap real}

The main new ingredient is Grisvard's trace characterization of real interpolation spaces~\cite[Thm.~5.12]{Grisvard-Interpolation} stated in Proposition~\ref{prop: Grisvard trace theorem} below.

For a Banach space $X$ we need the usual Bochner--Lebesgue space $\L^p(\R;X)$ of $X$-valued $p$-integrable functions on the real line and for $s > 0$ the respective (fractional) Sobolev spaces $\W^{s,p}(\R;X)$ that are defined as in the scalar case upon replacing absolute values by norms on $X$. For $s > 1/p$ such functions have a continuous representative and in that sense $\W^{s,p}(\R;X) \subseteq \C(\R; X)$ holds with continuous inclusion~\cite[Cor.~26]{Simon}. In particular, the pointwise evaluation $|_{t=0}: \W^{s,p}(\R;X) \to X$ is well-defined and bounded. All this was already used in \cite{Grisvard-Interpolation} and known at the time by different proofs. 

\begin{proposition}[Grisvard]
\label{prop: Grisvard trace theorem}
Let $X_0, X_1$ be Banach spaces such that $X_1 \subseteq X_0$ with dense and continuous inclusion. Let $p \in (1,\infty)$ and $s > 1/p$. Then
\begin{align*}
  \big(\L^p(\R; X_1) \cap \W^{s,p}(\R; X_0) \big)|_{t = 0} \subseteq \big(X_0,X_1\big)_{1-\frac{1}{sp}, p}.
\end{align*}
\end{proposition}

The strategy to obtain Theorem~\ref{thm: real interpolation via trace} is schematically displayed in Figure~\ref{fig: trace method}. Owing to Proposition~\ref{prop: easy inclusion} and the bounded inverse theorem, we only need to prove the set inclusion ``$\supseteq$'' in \eqref{C6}. 

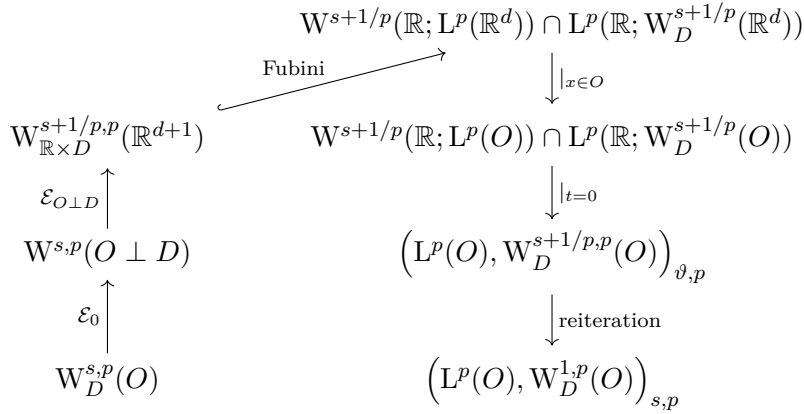
\begin{figure}[ht]
\centering
\[ \begin{tikzcd}
         & \W^{s+1/p}(\R; \L^p(\R^d)) \cap \L^p(\R; \W^{s+1/p}_D(\R^d)) \ar{d}{|_{x\in O}} \\
        \W^{s+1/p,p}_{\R \times D}(\R^{d+1}) \ar[hookrightarrow]{ru}{\text{Fubini}}             & \W^{s+1/p}(\R; \L^p(O)) \cap \L^p(\R; \W^{s+1/p}_D(O)) \ar{d}{|_{t = 0}}\\
        \W^{s,p}(O \perp D) \ar{u}{\Ext_{O \perp D}} & \left(\L^p(O), \W^{s+1/p,p}_D(O)\right)_{\vartheta,p} \ar{d}{\text{reiteration}} \\
        \W^{s,p}_D(O) \ar{u}{\Ext_0} & \left(\L^p(O), \W^{1,p}_D(O)\right)_{s,p}
\end{tikzcd}\]
\caption{Schematic presentation of the argument for obtaining the inclusion ``$\supseteq$'' in Theorem~\ref{thm: real interpolation via trace} for $s>1/p$. For $s<1/p$ the diagram would start with $\W^{s,p}(O)$ instead.}
\label{fig: trace method}
\end{figure}

The key observation is that functions in the starting space of Figure~\ref{fig: trace method} can be extended by zero to the set
\begin{align}
\label{def: D-cyl}
O \perp D := \big(O \times \{0\}\big) \cup \big(D \times \R\big),
\end{align}
without losing Sobolev regularity. We shall see that $O \perp D$ is, as expected, a $d$-regular subset of $\R^{d+1}$. By means of the Jonsson--Wallin operator $\Ext_{O \perp D}$ we can then extend to all of $\R^{d+1}$ and via a Fubini property we end up in a space suitable for Grisvard's result. Taking the trace yields the desired inclusion, up to applying reiteration techniques from Proposition~\ref{prop: reiteration real} in the final step.

Unless otherwise stated, we make the following 

\begin{assumption}
\label{ass: real}
The set $O\subseteq \R^d$ is open and $d$-regular. The Dirichlet part $D \subseteq \cl{O}$ is uniformly $(d-1)$-regular.
\end{assumption}

Only the final step will use the $(d-1)$-regularity of the full boundary $\bd O$ additionally assumed in Theorem~\ref{thm: real interpolation via trace}.

\subsection{Hardy's inequality}
\label{Subsec: Hardy}

In order to obtain the mapping property of the zero extension $\Ext_0$ in Figure~\ref{fig: trace method}, we establish a fractional Hardy inequality adapted to mixed boundary conditions that might be of independent interest. In contrast to related inequalities in \cite{Darmstadt-KatoMixedBoundary} we completely avoid the use of capacities.

\begin{definition}
\label{def: plump set}
A set $U \subseteq \R^d$ is \emph{plump} if there exists $\kappa \in (0,1)$ with the property:
\begin{align*}
 \forall x \in \cl{U}, r \leq \diam(U) \quad \exists y \in \B(x,r) \; : \; \B(y, \kappa r) \subseteq U.
\end{align*}
\end{definition}

\begin{remark}
\label{rem: plump set}
A comparison with Definition~\ref{def: porosity} yields first examples of plump sets. Namely, if $E \subseteq \R^d$ is uniformly porous, then ${}^c E$ is plump. This example can be modified to the effect that $E$ is bounded and (uniformly) porous and $Q \subseteq \R^d$ is an open cube containing $\cl{E}$: Still we have that $Q \setminus E$ is plump.
\end{remark}

We cite a result of Dyda--V\"ah\"akangas~\cite[Thm.~1]{Dyda-Vahakangas}.

\begin{proposition}[Dyda--V\"ah\"akangas]
\label{prop: DV}
Let $p \in (1,\infty)$ and $s \in (0,1)$. Suppose that $U \subseteq \R^d$ is a proper, plump, open set in $\R^d$. Assume one of the following conditions:
\begin{enumerate}
 \item $\cl{\dim}_{\AS}(\bd U) < d-sp$ and $U$ is unbounded.
 \item $\underline{\dim}_{\AS}(\bd U)>d-sp$ and either $U$ is bounded or $\bd U$ is unbounded.
\end{enumerate}
Then there exists a constant $c$ such that the inequality
\begin{align*}
 \int_U \frac{|f(x)|^p}{\dist(x, \bd U)^{sp}} \; \d x \leq c \int_U \int_U \frac{|f(x)-f(y)|^p}{|x-y|^{sp+d}} \; \d x \; \d y
\end{align*}
holds for all measurable functions $f$ for which the left-hand side is finite.
\end{proposition}

Upper and lower Assouad dimension have been introduced in Definition~\ref{def: Assouad dimension}. For uniformly $\ell$-regular sets they both equal $\ell$, see Proposition~\ref{prop: Assouad for Ahlfors}. With this at hand, we can state and prove the Hardy inequality alluded to above.

\begin{proposition}
\label{prop: fractional Hardy}
Let $p \in (1,\infty)$ and $s \in (0,1)$, $s \neq 1/p$. Under Assumption~\ref{ass: real} there is a constant $C>0$ such that the fractional Hardy inequality
\begin{align}
\label{eq: Hardy}
 \int_O \frac{|f(x)|^p}{\dist(x,D)^{sp}} \; \d x \leq C \|f\|_{\W^{s,p}(O)}^p
\end{align}
holds for all $f \in \W^{s,p}(O)$ if $s < 1/p$ and for all $f \in \W^{s,p}_D(O)$ if $s > 1/p$.
\end{proposition}

\begin{proof}
In both cases we shall reduce the claim to Proposition~\ref{prop: DV} on some auxiliary set.

\emph{Case 1: $s<1/p$}. Since $D$ is uniformly $(d-1)$-regular, so is $\cl{D}$. Due to Propositions~\ref{prop: Assouad for Ahlfors} and \ref{prop: porosity through assouad} we have that $\cl{D}$ is uniformly porous and hence $U \coloneqq \R^d \setminus \cl{D}$ is plump. Since $(d-1)$-regular sets have empty interior, we conclude $\bd U = \cl{D}$, which has upper and lower Assouad dimension $d-1$. Part (i) of Proposition~\ref{prop: DV} yields for all measurable $f$ for which the left-hand side is finite
\begin{align}
\label{eq1: Hardy}
 \int_{\R^d} \frac{|f(x)|^p}{\dist(x, D)^{sp}} \; \d x \leq c \int_{\R^d} \int_{\R^d} \frac{|f(x)-f(y)|^p}{|x-y|^{sp+d}} \; \d x \; \d y \leq c \|f\|_{\W^{s,p}(\R^d)}^p.
\end{align}
Example~\ref{ex: examples Dt class} guarantees that $x \mapsto \dist(x,D)^{-sp}$ is locally integrable. Thus, \eqref{eq1: Hardy} applies to every $f \in \C_0^\infty(\R^d)$, a dense subspace of $\W^{s,p}(\R^d)$, and we can use Fatou's lemma to extend \eqref{eq1: Hardy} to all $f \in \W^{s,p}(\R^d)$. Restriction to $O$ yields \eqref{eq: Hardy} for $f \in \W^{s,p}(O)$.

\emph{Case 2: $s > 1/p$ and $D$ is unbounded}. Part (ii) of Proposition~\ref{prop: DV} applies to $U \coloneqq \R^d \setminus \cl{D}$ and we can argue as before, except that now we have \eqref{eq1: Hardy} \emph{a priori} for $f \in \C_D^\infty(\R^d)$, a dense class of $f \in \W_D^{s,p}(\R^d)$ in view of Lemma~\ref{lem: testfunctions dense}. Hence we get \eqref{eq: Hardy} for $f \in \W_D^{s,p}(O)$.

\emph{Case 3: $s>1/p$ and $D$ is bounded}. Let $Q$ be an open cube that contains $\cl{D}$. As before we obtain that $U \coloneqq 2Q \setminus \cl{D}$ is plump, where $2Q$ denotes the concentric cube with twice the sidelength. Moreover, $\bd U = \bd (2Q) \cup \cl{D}$ is uniformly $(d-1)$-regular as a finite union of sets with that property. Hence, it has lower Assouad dimension $d-1$. Part (ii) of Proposition~\ref{prop: DV} yields for all measurable $g$ for which the left-hand side is finite
\begin{align}
\label{eq2: Hardy}
 \int_{U} \frac{|g(x)|^p}{\dist(x, \bd U)^{sp}} \; \d x 
 \leq c \int_{U} \int_{U} \frac{|g(x)-g(y)|^p}{|x-y|^{sp+d}} \; \d x \; \d y 
 \leq c \|g\|_{\W^{s,p}(U)}^p.
\end{align}
This applies to $g \in \C_{\bd U}^\infty(U)$ and extends to $g \in \W^{s,p}_{\bd U}(U)$ as in Case~2.

Let now $f \in \W^{s,p}_D(O)$. Let us fix $\eta \in \C_0^\infty(2Q)$ equal to $1$ on $Q$ and let $\Ext$ be an extension operator for $O$ as in Proposition~\ref{prop: Rychkov}. We can bound
\begin{align*}
\int_O \frac{|f(x)|^p}{\dist(x,D)^{sp}} \; \d x
&\leq  \int_{O \cap Q} \frac{|\eta(x) \Ext f(x)|^p}{\dist(x, D)^{sp}} \; \d x + \int_{O \setminus Q} \frac{|f(x)|^p}{\dist(x,D)^{sp}} \; \d x\\
&\leq  \int_{U} \frac{|\eta(x) \Ext f(x)|^p}{\dist(x, \bd U)^{sp}} \; \d x + \int_{O \setminus Q} \frac{|f(x)|^p}{\dist(x,D)^{sp}} \; \d x \eqqcolon I_1 + I_2,
\end{align*}
where we have used $O \cap Q \subseteq U$ and $D \subseteq \bd U$ to obtain $I_1$. Since on ${}^cQ$ we have $\dist(\cdot\,,D) \geq \dist(D, \bd Q) > 0$, we control $I_2 $ by $\|f\|_{\L^p(O)}^p$. Next, $\Ext f \in \W^{s,p}_D(\R^d)$ follows from Lemma~\ref{lem: Rychkov preserves boundary conditions}. Pointwise multiplication by $\eta$ is bounded on $\W^{s,p}(\R^d)$ and maps $\W_D^{s,p}(\R^d) \to \W^{s,p}_{\bd U}(\R^d)$ since this is true for the respective dense subsets provided by Lemma~\ref{lem: testfunctions dense}. In conclusion, we have $\eta \Ext f \in \W^{s,p}_{\bd U}(\R^d)$. Hence, \eqref{eq2: Hardy} gives control on $I_1$ by $\|\eta \Ext f\|_{\W^{s,p}(\R^d)}^p$. The boundedness of $\Ext$ leads us to a desirable bound for $I_1$.
\end{proof}

\subsection{Details of the proof}
\label{Subsec: Details}

We are in a position to give a precise meaning to Figure~\ref{fig: trace method}. We begin with the zero extension part on the left.

The following lemma is a straightforward consequence of a product formula for Hausdorff measure~ \cite[Thm.\ 2.10.45]{Federer}. Full details are written out in \cite[Cor.~7.6]{Darmstadt-KatoMixedBoundary}.

\begin{lemma}
\label{lem: Omega-Odot-D is a d-set}
If $O \subseteq \R^d$ is $d$-regular and $D \subseteq \R^{d}$ is $(d-1)$-regular, then $D \times \R$, $O \times \{0\}$, and $O \perp D \subseteq \R^d$ are $d$-regular.
\end{lemma}

Since $O \perp D$ is a $d$-regular subset of $\R^{d+1}$, the fractional Sobolev spaces $\W^{s,p}(O \perp D)$ can be defined as in Section~\ref{Subsec: Not Dirichlet} and there is a corresponding Jonsson--Wallin theory in $\R^{d+1}$. In the following, we systematically use bold face to distinguish geometric objects such as points, balls, and Hausdorff measures in $\R^{d+1}$ from their counterparts in $\R^d$.

\begin{proposition}
\label{prop: Zero extension to Omega-Odot-D}
Let $p \in (1,\infty)$ and $s \in (0,1)$, $s \neq 1/p$. Under Assumption~\ref{ass: real} the zero extension operator
\begin{align*}
(\Ext_0 f) (x,t) \coloneqq \begin{cases}
				    f(x) & (\text{if $x\in O$, $t=0$}) \\
				    0 & (\text{if $x \in D$, $t \in \R$})
                                 \end{cases}
\end{align*}
is $\W^{s,p}(O) \to \W^{s,p}(O \perp D)$-bounded  if $s<1/p$ and $\W^{s,p}_D(O) \to \W^{s,p}(O \perp D)$-bounded  if $s>1/p$.
\end{proposition}

\begin{proof}
Since the outer measure $E \mapsto \bm{\cH}^d(E \times \{0\})$ on $\R^d$ is a translation invariant Borel measure that assigns finite measure to the unit cube~\cite[\SS 27]{Yeh}, the induced measure coincides up to a norming constant $c_d > 0$ with the $d$-dimensional Lebesgue measure. Thus, $\Ext_0 f \in \L^p(O \perp D)$ is a consequence of $f \in \L^p(O)$. 

We use Tonelli's theorem to bound the remaining part of the $\W^{s,p}(O \perp D)$-norm by
\begin{align}
\label{eq1: zero extension}
\begin{split}
&\iint_{\substack{\bm{x}, \bm{y} \in O \perp D \\ |\bm{x} - \bm{y}| < 1}} \frac{|\Ext_0f(\bm{x}) - \Ext_0f(\bm{y})|^p}{|\bm{x}-\bm{y}|^{d+sp}} \; \bm{\cH}^d(\d \bm{x}) \; \bm{\cH}^d(\d \bm{y}) \\
&\leq c_d \iint_{\substack{x,y \in O \\ |x-y|<1}} \frac{|f(x) - f(y)|^p}{|x-y|^{d+sp}} \; \d x \; \d y \\
&\quad+2 \int_{O} \int_{\substack{\bm{y} \in D \times \R \\ |\bm{y} - (x,0)|<1}} \frac{|f(x)|^p}{|\bm{y} - (x,0)|^{d+sp}} \; \bm{\cH}^d(\d \bm{y}) \; \d x.
\end{split}
\end{align}
The first integral on the right is bounded by $\|f\|_{\W^{s,p}(O)}^p$. If the inner domain of integration in the second integral is non-empty, then there exists an integer $n_0 \geq 0$ such that $2^{-n_0-1} < \dist(x, D) \leq 2^{-n_0}$. We then split the integral into dyadic annuli
\begin{align*}
 \bm{C}_n:= \big(D \times \R\big) \cap \big((\bm{\B}((x,0), 2^{-n}) \setminus \bm{\B}((x,0), 2^{-n-1})\big), 
\end{align*}
each of which satisfies $\bm{\cH}^d(\bm{C}_n) \lesssim 2^{-dn}$ since $D \times \R$ is $d$-regular, to give
\begin{align*}
 \int_{\substack{\bm{y} \in D \times \R \\ |\bm{y} - (x,0)|<1}} \frac{1}{|\bm{y} - (x,0)|^{d+sp}} \; \bm{\cH}^d(\d \bm{y})
\lesssim \sum_{n=0}^{n_0} 2^{(n+1)(d+sp)} 2^{-dn}
=\frac{2^{d+sp}}{2^{sp}-1} (2^{sp(n_0+1)} -1).
\end{align*}
By choice of $n_0$, the right-hand side is controlled by $\dist(x,D)^{-sp}$. In conclusion, the second integral on the right of \eqref{eq1: zero extension} is bounded by
\begin{align*}
 \int_{O} \int_{\substack{\bm{y} \in D \times \R \\ |\bm{y} - (x,0)|<1}} \frac{|f(x)|^p}{|\bm{y} - (x,0)|^{d+sp}} \; \bm{\cH}^d(\d \bm{y}) \; \d x
 \lesssim \int_O \frac{|f(x)|^p}{\dist(x,D)^{sp}} \; \d x.
\end{align*}
The claim follows by Proposition~\ref{prop: fractional Hardy}.
\end{proof}

The Fubini property appearing in Figure~\ref{fig: trace method} is as follows. Throughout, we canonically identify $\L^p(\R^{d+1})$ with $\L^p(\R; \L^p(\R^d))$ by means of Fubini's theorem.

\begin{lemma}
\label{lem: Fubini property Wsp}
If $p \in (1,\infty)$ and $s \geq 0$, then up to equivalent norms
\begin{align}
\label{eq: Fubini property Wsp}
 \W^{s,p}(\R^{d+1}) = \L^p(\R; \W^{s,p}(\R^d)) \cap \W^{s,p}(\R; \L^p(\R^d)).
\end{align}
\end{lemma}

\begin{proof}
For $s \geq 0$ an integer, the claim follows directly from Fubini's theorem. Let now $s = k +\sigma$, where $k \geq 0$ is an integer and $\sigma \in (0,1)$. According to \cite[Sec.~2.5.1]{Triebel} we can equivalently norm $\W^{s,p}(\R^d)$ by
\begin{align*}
 \|f\| \coloneqq \|f\|_{\L^p} + \sum_{j = 1}^d \bigg(\int_{\R} \int_{\R^d} \bigg|\frac{\partial^k f}{\partial x_j^k}(x+he_j) - \frac{\partial^k f}{\partial x_j^k}(x)\bigg|^p \d x \; \frac{\d h}{|h|^{1+sp}} \bigg)^{1/p},
\end{align*}
where $(e_j)_j$ denote the standard unit vectors in $\R^d$. This equivalent norm only takes into account differences of $f$ along the coordinate axes. Therefore we obtain \eqref{eq: Fubini property Wsp} from Fubini's theorem if we equivalently norm all appearing spaces as described before.
\end{proof}

The next lemma makes Figure~\ref{fig: trace method} precise, except for the final step.

\begin{lemma}
\label{lem: interpolation non-matching parameters}
Let $p \in (1,\infty)$ and $s \in (0,1)$, $s \neq 1/p$. Under Assumption~\ref{ass: real} the set inclusion 
\begin{align*}
 (\L^p(O), \W^{s+1/p,p}_D(O))_{\vartheta,p} \supseteq \begin{cases}
                                                          \W_D^{s,p}(O) &(\text{if $s > 1/p$}) \\
							  \W^{s,p}(O) &(\text{if $s < 1/p$})
                                                         \end{cases},
\end{align*}
holds for $\vartheta \in (0,1)$ satisfying $\vartheta(s+1/p) = s$.
\end{lemma}

\begin{proof}
We fix $f$, which is a function in $\W^{s,p}(O)$ if $s<1/p$ and in $\W_D^{s,p}(O)$ if $s>1/p$, respectively. 

The inclusion $\W_D^{s+1/p,p}(O) \subseteq \L^p(O)$ is continuous and it is dense since already $\C_{\bd O}^\infty(O)$ is dense in $\L^p(O)$. In view of Proposition~\ref{prop: Grisvard trace theorem} it suffices to construct a function
\begin{align}
\label{Interpolation with non-matching parameters: desired properties of F}
 F \in \L^p(\R; \W_D^{s+1/p,p}(O)) \cap \W^{s+1/p, p}(\R; \L^p(O)) \quad \text{such that} \quad F|_{t=0} = f.
\end{align}
For the construction we start by extending $f$ to $O \perp D$ by zero. This extension $\Ext_0 f$ is in $\W^{s,p}(O \perp D)$ due to Proposition~\ref{prop: Zero extension to Omega-Odot-D}. Since $O \perp D$ is a $d$-regular subset of $\R^{d+1}$ according to Lemma~\ref{lem: Omega-Odot-D is a d-set}, we can use Proposition~\ref{prop: JW} to extend $\Ext_0 f$ to a function $G \in \W^{s+1/p,p}(\R^{d+1})$ in virtue of the corresponding Jonsson--Wallin operator. In view of Lemma~\ref{lem: Fubini property Wsp} we have by canonical identification
\begin{align}
\label{eq1: interpolation non matching parameters}
 G \in \L^p(\R; \W^{s+1/p,p}(\R^d)) \cap \W^{s+1/p,p}(\R; \L^p(\R^d)).
\end{align}
A closer inspection reveals the following.
\begin{enumerate}
  \item Let $\Res$ be the Jonsson--Wallin restriction to the $d$-set $D \times \R$ in $\R^{d+1}$. We have $\Res G = 0$ by construction and therefore $G \in \W^{s+1/p,p}_{D \times \R}(\R^{d+1})$. We introduce $t \coloneqq \min\{s+1/p, 1\}$ to have Lemma~\ref{lem: testfunctions dense} at our disposal and approximate $G \in \W^{t,p}_{D \times \R}(\R^{d+1})$ in that space by test functions $G_n \in \C_{D \times \R}^\infty(\R^{d+1})$. By slicing
  \begin{align*}
   G_n \in \L^p(\R; \W_D^{t,p}(\R^d)) \cap \W^{t,p}(\R; \L^p(\R^d))
  \end{align*}
  and due to the Fubini property of Lemma~\ref{lem: Fubini property Wsp} the limit $G$ is contained in the same space. From consistency of the restriction operator $\Res$ on fractional Sobolev spaces we can infer $\W_D^{t,p}(\R^d) \cap \W^{s+1/p,p}(\R^d) = \W^{s+1/p,p}_D(\R^d)$. This being said, it follows from \eqref{eq1: interpolation non matching parameters} that we have
  \begin{align*}
  G \in \L^p(\R; \W^{s+1/p,p}_D(\R^d)) \cap \W^{s+1/p,p}(\R; \L^p(\R^d)). 
  \end{align*}

  \item Let $\Res$ be the Jonsson--Wallin restriction to the $d$-set $\R^d \times \{0\}$ in $\R^{d+1}$. This operator is bounded from $\W^{s+1/p}(\R^{d+1})$ into $\L^p(\R^d \times \{0\})$ by Proposition~\ref{prop: JW}. On the other hand, we can look at the restriction $|_{t=0}$ defined on $\L^p(\R; \W^{s+1/p,p}(\R^d)) \cap \W^{s+1/p,p}(\R; \L^p(\R^d))$ and bounded into $\L^p(\R^d)$. Identifying corresponding objects via Fubini's theorem as before, it turns out that these two restrictions are the same since they obviously agree on a dense class of continuous functions. Since $\Res G$ and $f$ coincide $\bm{\cH}^d$-almost everywhere on $O \times \{0\}$ by construction, we can record
  \begin{align*}
   G|_{t=0} = f \quad \text{almost everywhere on $O$}.
  \end{align*}
\end{enumerate}
The outcome of observations (i) and (ii) shows that $F \coloneqq G|_{O \times \R}$ verifies \eqref{eq1: interpolation non matching parameters}.
\end{proof}

Together with Proposition~\ref{prop: easy inclusion} we obtain

\begin{corollary}
\label{cor: interpolation with non matching parameters}
If in addition to Assumption~\ref{ass: real} the set $O$ has $(d-1)$-regular boundary, then the set inclusion in Lemma~\ref{lem: interpolation non-matching parameters} is an equality with equivalent norms.
\end{corollary}

Eventually, we can complete the

\begin{proof}[Proof of Theorem~\ref{thm: real interpolation via trace}]
In the following all function spaces will be on $O$ and we omit the dependence on $O$ for clarity. In view of the reiteration theorem above it suffices to treat the case $s_0 = 0$ and $s_1 = 1$ and prove for $s\in (0,1)$ that up to equivalent norms it follows that
\begin{align}
\label{eq: goal real interpolation}
(\L^p, \W_D^{1,p})_{s, p} &= \begin{cases} 
                                               \W_D^{s,p} &(\text{if $s>1/p$})\\ \W^{s,p}  &(\text{if $s<1/p$})
					      \end{cases}.
\end{align}
If $s+1/p = 1$, then the claim follows from Corollary~\ref{cor: interpolation with non matching parameters}. The proof for $s +1/p \neq 1$ divides into four cases.

\emph{Case 1: $s>1/p$ and $s+1/p < 1$}. We have a diagram suitable for Wolff interpolation:
\begin{center}
\begin{tikzpicture}[scale=0.7]
\coordinate[label=left:$\L^p$] (X0) at (0,0);
\coordinate[label=left:$\W^{s,p}_D$] (Xtheta) at (3,0);
\coordinate[label=right:$\W^{s+1/p,p}_D$] (Xeta) at (5,0);
\coordinate[label=right:$\W^{1,p}_D.$] (X1) at (9,0);
\coordinate[label=above:${(\cdot\,, \cdot)_{\vartheta,p}}$] (I1) at (3,0.7);
\coordinate[label=below:${(\cdot\,, \cdot)_{\mu,p}}$] (I1) at (5,-0.7);

\draw (0,0) -- (0,0.7);
\draw (0,0.7) -- (5,0.7);
\draw (5,0.7) -- (5,0);
\draw[dashed] (3,0) -- (3,0.7);

\draw (3,0) -- (3,-0.7);
\draw (3,-0.7) -- (9,-0.7);
\draw (9,-0.7) -- (9,0);
\draw[dashed] (5,-0.7) -- (5,0);

\fill (X0) circle (3pt);
\fill (Xtheta) circle (3pt);
\fill (Xeta) circle (3pt);
\fill (X1) circle (3pt);
\end{tikzpicture}
\end{center}
Indeed, the $(\vartheta,p)$-interpolation is due to Corollary~\ref{cor: interpolation with non matching parameters} and the $(\mu,p)$-interpolation with suitable $\mu \in (0,1)$ is due to Theorem~\ref{thm: main result large s}. The claim follows by Proposition~\ref{prop: Wolff}.

\emph{Case 2: $s>1/p$ and $s+1/p > 1$}. We fix any $t \in (s,1)$ and let $\lambda \in (0,1)$ satisfy $(1-\lambda)s + \lambda(s+1/p) = t$. Applying one after the other Theorem~\ref{thm: main result large s}, Corollary~\ref{cor: interpolation with non matching parameters}, and Proposition~\ref{prop: reiteration real}, we obtain
\begin{align*}
 \W^{t,p}_D = (\W^{s,p}_D, \W^{s+1/p,p}_D)_{\lambda,p} = ((\L^p, \W^{s+1/p,p}_D)_{\vartheta, p},  \W^{s+1/p,p}_D)_{\lambda,p} = (\L^p, \W^{s+1/p,p})_{\theta,p},
\end{align*}
with $\theta = t/(s+1/p)$. Once again by Proposition~\ref{prop: reiteration real} and Corollary~\ref{cor: interpolation with non matching parameters} we find
\begin{align*}
 (\L^p, \W^{t,p}_D)_{s/t,p} = (\L^p, (\L^p, \W^{s+1/p,p}_D)_{\theta,p})_{s/t,p} = (\L^p, \W^{s+1/p,p}_D)_{\vartheta,p} = \W^{s,p}_D.
\end{align*}
Thus we obtain the desired result \eqref{eq: goal real interpolation} from Proposition~\ref{prop: Wolff} applied as follows:
\begin{center}
\begin{tikzpicture}[scale=0.7]
\coordinate[label=left:$\L^p$] (X0) at (0,0);
\coordinate[label=left:$\W^{s,p}_D$] (Xtheta) at (3,0);
\coordinate[label=right:$\W^{t,p}_D$] (Xeta) at (5,0);
\coordinate[label=right:$\W^{1,p}_D.$] (X1) at (9,0);
\coordinate[label=above:${(\cdot\,, \cdot)_{s/t,p}}$] (I1) at (3,0.7);
\coordinate[label=below:${(\cdot\,, \cdot)_{\mu,p}}$] (I1) at (5,-0.7);

\draw (0,0) -- (0,0.7);
\draw (0,0.7) -- (5,0.7);
\draw (5,0.7) -- (5,0);
\draw[dashed] (3,0) -- (3,0.7);

\draw (3,0) -- (3,-0.7);
\draw (3,-0.7) -- (9,-0.7);
\draw (9,-0.7) -- (9,0);
\draw[dashed] (5,-0.7) -- (5,0);

\fill (X0) circle (3pt);
\fill (Xtheta) circle (3pt);
\fill (Xeta) circle (3pt);
\fill (X1) circle (3pt);
\end{tikzpicture}
\end{center}
Indeed, we have obtained the $(s/t,p)$-interpolation above and the $(\mu,p)$-interpolation for appropriately chosen $\mu$ is due to Theorem~\ref{thm: main result large s}. Note that because of the exceptional case for real interpolation of Sobolev spaces we cannot pick $t=1$ right away.

\emph{Case 3: $s<1/p$ and $s+1/p < 1$}. We can apply one of the previous two cases with $s+1/p$ in place of $s$ to obtain $(\L^p, \W^{1,p}_D)_{s+1/p,p} = \W^{s+1/p,p}_D$. Together with Corollary~\ref{cor: interpolation with non matching parameters} in the first and reiteration in the third step, we are led to the desired result
\begin{align*}
 \W^{s,p} = (\L^p, \W^{s+1/p,p}_D)_{\vartheta,p} = (\L^p, (\L^p, \W^{1,p}_D)_{s+1/p,p})_{\vartheta,p} = (\L^p, \W^{1,p}_D)_{s,p}.
\end{align*}

\emph{Case 4: $s<1/p$ and $s+1/p > 1$}. We pick $1/p < \lambda < \kappa < 1$. By one of the first two cases along with Proposition~\ref{prop: reiteration real}, we find
\begin{align*}
 (\L^p, \W^{\kappa,p}_D)_{\lambda/\kappa,p} = (\L^p, (\L^p, \W^{1,p}_D)_{\kappa,p})_{\lambda/\kappa,p} = (\L^p, \W^{1,p}_D)_{\lambda, p} = \W^{\lambda, p}_D.
\end{align*}
Together with Theorem~\ref{thm: main result large s} this establishes for suitable $\mu$ the diagram
\begin{center}
\begin{tikzpicture}[scale=0.7]
\coordinate[label=left:$\L^p$] (X0) at (0,0);
\coordinate[label=left:$\W^{\lambda,p}_D$] (Xtheta) at (3,0);
\coordinate[label=right:$\W^{\kappa,p}_D$] (Xeta) at (5,0);
\coordinate[label=right:$\W^{s+1/p,p}_D$.] (X1) at (9,0);
\coordinate[label=above:${(\cdot\,, \cdot)_{\lambda/\kappa,p}}$] (I1) at (3,0.7);
\coordinate[label=below:${(\cdot\,, \cdot)_{\mu,p}}$] (I1) at (5,-0.7);

\draw (0,0) -- (0,0.7);
\draw (0,0.7) -- (5,0.7);
\draw (5,0.7) -- (5,0);
\draw[dashed] (3,0) -- (3,0.7);

\draw (3,0) -- (3,-0.7);
\draw (3,-0.7) -- (9,-0.7);
\draw (9,-0.7) -- (9,0);
\draw[dashed] (5,-0.7) -- (5,0);

\fill (X0) circle (3pt);
\fill (Xtheta) circle (3pt);
\fill (Xeta) circle (3pt);
\fill (X1) circle (3pt);
\end{tikzpicture}
\end{center}
Proposition~\ref{prop: Wolff} yields $\W^{\kappa,p}_D = (\L^p, \W^{s+1/p,p}_D)_{\theta,p}$ with $\theta = \kappa/(s+1/p)$. We conclude by using one after the other Corollary~\ref{cor: interpolation with non matching parameters}, reiteration, one of the first two cases, and again reiteration:
\begin{align*}
 \W^{s,p} &= (\L^p, \W^{s+1/p,p}_D)_{\vartheta,p} = (\L^p, (\L^p, \W^{s+1/p,p}_D)_{\theta,p})_{s/\kappa,p} = (\L^p, \W^{\kappa,p}_D)_{s/\kappa,p} \\
	  &= (\L^p, (\L^p, \W^{1,p}_D)_{\kappa,p})_{s/\kappa,p} = (\L^p, \W^{1,p}_D)_{s,p}. \qedhere
\end{align*}
\end{proof}
\appendix
\section{Porous sets}
\label{Sec: Porous sets}

We provide a streamlined approach to the geometry of porous sets. All this is known to the experts but some results require going through existing literature in a rather opaque way. The reader may look up relevant definitions in Section~\ref{Subsec: Not Geometry}.

\begin{lemma}
\label{lem: porous boundary null set}
    Every porous set $E\subseteq \R^{d}$ is a Lebesgue null set.
\end{lemma}

\begin{proof}
By Remark~\ref{rem: porosity}, each ball $B$ centered in $E$ contains a ball of comparable radius that does not intersect $E$. Hence, there is $\delta \in (0,1)$ depending only on $E$ such that
 \begin{align*}
  \frac{|B\cap E|}{|B|} \leq 1-\delta.
 \end{align*}
By Lebesgue's differentiation theorem this implies $\mathds{1}_E = 0$ almost everywhere.
\end{proof}

We recall the Vitali covering lemma that will be used frequently in the following, see \cite[Thm. 1.2]{Heinonen}.

\begin{lemma}
\label{lem: Vitali}
    Let $\{B_i\}_{i\in I}$ be a family of open balls with uniformly bounded radii. Then there exists a subfamily $\{B_j\}_{j\in J}$ of disjoint balls such that $$\bigcup_i B_i \subseteq \bigcup_j 5 B_j.$$
\end{lemma}


\begin{corollary} 
\label{cor: cover subset of Rd by balls}
    Let $E\subseteq \R^d$ and $0<r\leq R<\infty$. For any ball $B$ of radius $R$ the set $E\cap B$ can be covered by $10^d(R/r)^d$ ball of radius $r$ centered in $E\cap B$.
\end{corollary}

\begin{proof}
   Consider the covering $\{\B(x,r/5)\}_{x\in B\cap E}$ of $B \cap E$. We find a disjoint subfamily $\{B_i\}_{i\in I}$ such that $B\cap E \subseteq \cup_{i\in I} 5B_i$. We denote by $\#_i$ the cardinality of $I$ and calculate
    \begin{align*}
        \#_i c_d (r/5)^d = |\cup_{i\in I} B_i| \leq |2B| = c_d 2^d R^d,
    \end{align*}
    where $c_d$ is the measure of the unit ball. This shows $\#_i \leq 10^d (R/r)^d$.
\end{proof}

We continue with the simple observation that the radius bound by $1$ in the definition of $\ell$-regularity is arbitrary.

\begin{lemma}
\label{lem: Ahlfors radius bound}
    Let $E\subseteq \R^d$ and $0<\ell\leq d$. If for some $M \in (0,\infty)$ there is comparability $\cH^\ell(B \cap E) \approx \r(B)^\ell$ uniformly for all open balls $B$ of radius $\r(B) \leq M$ centered in $E$, then the same is true for any $M \in (0,\infty)$.
\end{lemma}

\begin{proof}
    Suppose we have uniform comparability for balls up to radius $\r(B) \leq m$. Given $M > m$, we need to extend it to balls $B$ centered in $E$ of radius $\r(B) \leq M$. Let $c \coloneqq m/M$. The calculation
    \begin{align*}
        \frac{m^\ell \r(B)^\ell}{M^\ell} \lesssim \cH^l(cB \cap E) \leq \cH^\ell(B\cap E)
    \end{align*}
    gives the lower estimate. For the upper one, we cover $B\cap E$ by $10^d/c^d$ balls of radius $c \r(B)$ centered in $B \cap E$ according to Corollary~\ref{cor: cover subset of Rd by balls} and conclude $\cH^\ell(B\cap E) \lesssim \r(B)^\ell$.
\end{proof}

We come to computing the Assouad dimensions of Ahlfors-regular sets.

\begin{lemma} \label{lem: coverings for Assouad}
    Let $E\subseteq \R^d$ be $\ell$-regular for some $0<\ell \leq d$ and let $M<\infty$. There exist constants $c, C> 0$ such that, if $x \in E$ and $0 < r\leq R < M$, then in order to cover $E \cap\B(x,R)$ by balls of radius $r$ centered in $E$, at least $c(R/r)^\ell$ and at most $C(R/r)^\ell$ balls are needed. If $E$ is unbounded and uniformly $\ell$-regular, then this also holds for $M = \infty$.
\end{lemma}

\begin{proof}
    Let $\{B_i\}_{i\in I}$ be \emph{some} cover of $E\cap\B(x,R)$ by balls of radius $r$. We use Lemma~\ref{lem: Ahlfors radius bound} to calculate
    \begin{align}
        R^\ell \lesssim \cH^\ell(B(x,R) \cap E) \leq \cH^\ell(\cup_{i\in I} B_i \cap E) \leq \sum_{i\in I} \cH^\ell(B_i\cap E) \lesssim \#_i r^\ell,
    \end{align}
    which shows $\#_i \gtrsim (R/r)^\ell$ and gives the constant $c$. As for $C$, we select a subfamily of disjoint balls $B_j$ from the covering $\{\B(x,r/5)\}_{x\in B\cap E}$ of $B\cap E$. Then we estimate, using Lemma~\ref{lem: Ahlfors radius bound},
    \begin{align}
    \label{eq: counting disjointed families}
        \#_j (r/5)^\ell \lesssim \sum_{j\in J} \cH^\ell(B_j \cap E) \leq \cH^\ell(2B \cap E) \lesssim (2R)^\ell
    \end{align}
    and conclude $\#_j \lesssim (R/r)^\ell$.
\end{proof}

\begin{proposition}
\label{prop: Assouad for Ahlfors}
    Let $E\subseteq \R^d$ be uniformly $\ell$-regular. It follows that $\underline{\dim}_\AS(E)= \overline{\dim}_\AS(E) =\ell$.
\end{proposition}

\begin{proof}
    We can rephrase Lemma~\ref{lem: coverings for Assouad} in the language of Definition~\ref{def: Assouad dimension}. It precisely asserts that $\ell \in \underline{\AS}(E) \cap \overline{\AS}(E)$. Hence, we get $\underline{\dim}_{\AS}(E) \geq \ell$ and  $\overline{\dim}_{\AS}(E) \leq \ell$. The claim follows since $\underline{\dim}_{\AS}(E) \leq \overline{\dim}_{\AS}(E)$ holds for any set $E$. Indeed, given $\lambda \in \underline{\AS}(E)$ and $\mu \in \overline{\AS}(E)$ we have $(R/r)^\lambda \lesssim (R/r)^\mu$ for all $0<r<R<\diam(E)$ and hence $\lambda \leq \mu$.
%
\end{proof}

We turn to porosity. The following result was already mentioned in Section~\ref{Subsec: Intro geometry}.

\begin{lemma} \label{lem: lower dimensional implies porous}
    Let $E \subseteq F\subseteq \R^d$. If $F$ is $\ell$-regular and $E$ is $m$-regular with $0<m<\ell\leq d$, then $E$ is porous in $F$. Likewise, if $\overline{\dim}_{\AS}(E) < \underline{\dim}_{\AS}(F)$, then $E$ is uniformly porous in $F$.
\end{lemma}

\begin{proof}
    We begin with the first claim. Lemma~\ref{lem: coverings for Assouad} yields some $C\geq 1$ such that, if $x \in E$ and  $0<r\leq R\leq 1$, then at most $C(2R/r)^m$ balls of radius $r$ centered in $E$ are needed to cover $E \cap\B(x,2R)$. It also yields some $c > 0$ such that at least $c(R/(2r))^\ell$ balls of radius $2r$ centered in $F$ are needed to cover $F\cap\B(x,R)$. We use this observation with $r = \kappa R$, where $\kappa \in (0,1)$ satisfies $c/(2\kappa)^\ell > C(2/\kappa)^m$. This is possible due to $m < \ell$. 
    
    Let $\{B_i\}_{i \in I}$ be a family of $\#_i \leq C(2/\kappa)^m$ balls of radius $r$ centered in $E$ that cover $E \cap\B(x,2R)$. By choice of $\kappa$ the balls $\{2B_i\}_{i \in I}$ cannot cover $F \cap\B(x,R)$. Pick $y \in F \cap\B(x,R)$ that is not contained in any of the $2B_i$. By construction we have $\B(y,r) \subseteq \R^d \setminus \cup_{i \in I} B_i$ but due to $r < R$ we also have $\B(y,r) \subseteq \B(x, 2R)$ and hence $E \cap \B(y,r) \subseteq \cup_{i} B_i$. Thus, we must have $E \cap \B(y,r) = \emptyset$ and conclude that $E$ is porous in $F$.
    
    The proof of the second claim is identical, but we do not assume $R \leq 1$ and have the covering properties for some $m \in \cl{\AS}(E)$ and $\ell \in \underline{\AS}(F)$ with $m < \ell$ by assumption.
\end{proof}

\begin{lemma} \label{lem: porosity implies covering property}
    If $E\subseteq \R^d$ is porous, then there exist $C\geq 1$ and $0<s<d$ such that, given $x \in E$ and $0 < r < R\leq 1$, there is a covering of $E\cap B(x,R)$ by $C(R/r)^s$ balls of radius $r$ centered in $E$. Moreover, if $E$ is uniformly porous, then $\cl{\dim}_\AS(E)<d$.
\end{lemma}

\begin{proof}
    We only show the porous case since the uniform case again just follows by dropping all restrictions on the radii. In the following all cubes are closed and axis-aligned. As indicated in Section~\ref{Subsec: Not Geometry}, we can equivalently replace balls by cubes and radii by side lengths in the definition of porosity and Assouad dimension. Likewise, it suffices to establish the claim of the lemma with cubes. 
    
    In view of Remark~\ref{rem: porosity} we can fix $n\in \IN$ such that for every cube $Q \subseteq \R^d$ there is a cube $Q' \subseteq Q \setminus E$ of sidelength $\ell(Q') = \ell(Q)/n$. We fix a cube $Q$ centered in $E$ of side length $R \leq 1$. Let $0<r\leq R$ and fix $k\in \IN$ such that $R/(2n)^{k+1} \leq r < R/(2n)^k$. We claim that we can cover $Q$ by $((2n)^d-1)^{k+1}$ closed cubes of side length $R/(2n)^{k+1}$. Put $s\coloneqq \log((2n)^d-1)/\log(2n)<d$. Then
    \begin{align}
        ((2n)^d-1)^{k+1} = (2n)^s (2n)^{ks} < (2n)^s (R/r)^s
    \end{align}
    shows the assertion.

    For the claim we start with $k=1$. There is a cube $Q' \subseteq Q\setminus E$ of side length $R/n$. Then there is a cube $Q''$ in the grid of $(2n)^{d}$ cubes with sidelength $R/(2n)$ covering $Q$ that is contained in $Q'$. This means that we only need $(2n)^d-1$ cubes of side length $R/(2n)$ to cover $E$. We conclude by applying this argument inductively on each cube of the previous covering.
\end{proof}

Combining the uniform cases of the two preceding lemmas lets us re-obtain a result of Luukkainen~\cite[Thm~5.2]{Luukkainen}. Note that $\underline{\dim}_{\AS}(\R^d) = d$ due to Proposition~\ref{prop: Assouad for Ahlfors}.

\begin{proposition}
\label{prop: porosity through assouad}
A set $E \subseteq \R^d$ is uniformly porous if and only if its upper Assouad dimension is strictly less than $d$.
\end{proposition}

We can use the non-uniform cases to show that some open sets are of class $\cD^t$. The argument is a slight adaption of \cite[Thm.~4.2]{Lehrbaeck-Tuominen_NoteOneDimensions}.

\begin{proposition}
\label{prop: examples Dt}
    Let $O\subseteq \R^d$ be open. If $\bd O$ is porous, then $O \in \cD^t$ for some $t\in (0,1)$. If $\bd O$ is $\ell$-regular for some $0<\ell<d$, then $O \in \cD^t$ for all $t \in (0, \max\{1, d -\ell\})$.
\end{proposition}

\begin{proof}
    If $\bd O$ is porous, then we pick $C\geq 1$ and $0<s<d$ according to Lemma~\ref{lem: porosity implies covering property} such that for each $j\geq 0$ and for any ball $B$ with radius $r\leq 1$ centered in $\bd O$ we can cover $B\cap \bd O$ by at most $C2^{js}$ balls of radius $r2^{-j}$. If $\bd O$ is $\ell$-regular, then Lemma~\ref{lem: coverings for Assouad} guarantees that we can take $s = \ell$. In any case, fix $\max(s,d-1)<u<d$. Put $E_j\coloneqq \{x\in B: \dist(x,\bd O) \leq r2^{-j}\}$ and $A_j\coloneqq E_j \setminus E_{j+1}$. By construction, the covering property for $B\cap \bd O$ implies that we can cover $E_j$ by at most $C2^{js}$ balls of radius $r2^{-(j-1)}$. The $d$-regularity of Lebesgue measure then implies
    \begin{align}
        |A_j| \leq |E_j| \lesssim 2^{js} r^d 2^{d-jd}. \label{eq: measure Aj}
    \end{align}
    We use that $\{A_j\}_{j\geq 0}$ is a disjoint cover of $B \setminus \bd O$, comparability $\dist(x,\bd O) \approx r2^{-j}$ on $A_j$, estimate \eqref{eq: measure Aj}, and $s<u$ to calculate
    \begin{align}
        \int_{B\setminus \bd O} \dist(y,\bd E)^{u-d} \d y &\leq \sum_j \int_{A_j} \dist(y,\bd O)^{u-d} \d y
                                                             \lesssim \sum_j |A_j| 2^{dj-u j} r^{u-d}\\
                                                             &\lesssim \sum_j r^u 2^{j(s-u)} \lesssim r^u.
    \end{align}
    Setting $t \coloneqq d-u\in (0,1)$, we write this in the form
    \begin{align}
        \sup_{x\in \bd O} \sup_{0<r\leq 1} r^{t-d} \int_{\B(x,r)\setminus \bd O} \dist(y,\bd O)^{-t} < \infty,
    \end{align}
    which just means that $O \in \cD^t$. In the case of $\ell$-regular boundary, every $u \in (\max\{\ell,d-1\},1)$ and thus every $t \in (0,\max\{1, d-\ell\})$ was admissible in the proof.  
\end{proof}

\section{Intrinsic characterizations}
\label{Sec: Intrinsic characterizations}

Although perfectly suited for interpolation questions, the function spaces $\X^{s,p}(O) = \X^{s,p}(\R^d)|_O$ lack an \emph{intrinsic characterization} through a norm that only uses information on $O$. The problem of finding such characterizations has a long history and we refer for instance to \cite{Hajlasz-Koskela-Tuominen, Shvartsman, Rychkov, Triebel-intrinsic, Triebel, ET, JW,Zhou} and references therein. Here, we only mention two results that are of particular importance for putting our paper into context of work on mixed boundary value problems~\cite{RobertJDE, TripleMitrea-Brewster, Disser, Disser-Meyries-Rehberg,HJKR, Rehberg-terElstADE, ABHR}.

The following is the full-dimensional case in \cite[Thm.VI.1]{JW}.

\begin{proposition}
\label{prop: JW trace on dset}
Let $O \subseteq \R^d$ be an open $d$-regular set and let $p \in (1,\infty)$ and $s \in (0,1)$. Then $\W^{s,p}(O)$ is up to equivalent norms the space of those $f \in \L^p(O)$ for which
\begin{align*}
 \|f\| \coloneqq \|f\|_{\L^p(O)} + \bigg(\iint_{O \times O} \frac{|f(x) - f(y)|^p}{|x-y|^{d+ sp}} \; \d x \; \d y \bigg)^{1/p} < \infty.
\end{align*}
\end{proposition}

\begin{remark}
\label{rem: JW trace on dset}
If in the setting above $D \subseteq \cl{O}$ is $(d-1)$-regular, then $\W^{s,p}_D(O)$ is the closure of $\C_D^\infty(O)$ for the intrinsic norm $\|\cdot\|$. This follows from Lemma~\ref{lem: testfunctions dense}.
\end{remark}

Probably most important result of the above type concerns $\W^{1,p}_D(O)$.

\begin{proposition}
\label{prop: W1pE local definition}
Suppose $O \subseteq \R^d$ is an open set, $D \subseteq \cl{O}$ is $(d-1)$-regular, and $O$ satisfies a uniform Lipschitz condition around $\cl{\partial O \setminus D}$. Then $\W^{1,p}_D(O)$ can equivalently be normed by
\begin{align*}
 \|f\| \coloneqq \Big( \|f\|_{\L^p(O)}^p + \|\nabla f\|_{\L^p(O)}^p \Big)^{1/p}
\end{align*}
and in fact it is the closure of $\C_D^\infty(O)$ for this norm.
\end{proposition}

\begin{proof}
Let $X$ be the closure of $\C_D^\infty(O)$ for the norm $\|\cdot\|$. Clearly we have $\W^{1,p}_D(O) \subseteq X$ with continuous inclusion. For the converse inclusion we argue as in \cite[Lem.~3.2]{ABHR}, using the localization formalism of Section~\ref{Subsec: Localization} and in particular the maps $\Ext$ and $\Res$ defined in \eqref{localization E} and \eqref{localization R}. We let $\sigma$ be the extension of functions from $(0,1) \times (-1,1)^{d-1}$ to $(-1,1)^{d}$ by even reflection and consider the operator
\begin{align*}
 \mathcal{T}: f \mapsto \Res ((\Ext f)_0, (\sigma (\Ext f)_i)_{i \in I}).
\end{align*}
By construction $\mathcal{T}f|_O = f$ holds for $f \in X$. It follows from the proof of Lemma~\ref{lem: E and R localization} that $\mathcal{T}: X \to \W^{1,p}(\R^d)$ is bounded. The proof of Lemma~\ref{lem: E and R localization BC} reveals that $\mathcal{T}$ in fact maps $X$ into $\W^{1,p}_D(\R^d)$, which implies $X \subseteq \W^{1,p}_D(O)$.
\end{proof}

\end{document}